\documentclass[11pt]{amsart}
\usepackage{amssymb}
\usepackage{amsmath}
\usepackage{amsfonts,latexsym,amstext}
\usepackage[foot]{amsaddr}

\usepackage{color}
\usepackage{bbm}

\newtheorem{theorem}{Theorem}[section]
\newtheorem{definition}[theorem]{Definition}

\newtheorem{lemma}[theorem]{Lemma}
\newtheorem{proposition}[theorem]{Proposition}
\newtheorem{corollary}[theorem]{Corollary}
\newtheorem{remark}[theorem]{Remark}
\newtheorem{assump}[theorem]{Assumptions}

\newcommand{\R}{{\mathbb R}}
\newcommand{\C}{{\mathcal C}}

\newcommand{\Z}{{\mathbb{Z}}}

\newcommand{\N}{{\mathbb N}}

\newcommand{\F}{{\mathcal F}}

\newcommand{\E}{{\mathbb E}}

\newcommand{\B}{{\mathcal B}}

\def\R{\mathbb{R}}
\def\N{\mathbb{N}}
\def\E{\mathbb{E}}

\def\F{\mathcal F}

\def\L{\mathcal{L}}

\title{On the Kesten-Stigum theorem in $L^2$ beyond $\lambda$-positivity}

\author{Matthieu Jonckheere$^*$ and Santiago Saglietti$^\dagger$}

\address[$*$]{Departamento de Matem\'atica, Universidad de Buenos Aires, 1428 Buenos Aires, Argentina.\newline email:\,{\tt mjonckhe@dm.uba.ar.}}

\address[$\dagger$]{Faculty of Industrial Engineering and Management, Technion, 3200003 Haifa, Israel.\newline email:\,{\tt saglietti.s@technion.ac.il.}} 

\begin{document}
\maketitle

\begin{abstract}
We study supercritical branching processes in which all particles evolve according to some general Markovian motion (which may possess absorbing states) and branch independently at a fixed constant rate. Under fairly natural assumptions on the asymptotic distribution of the underlying motion,
we first show using only probablistic tools that there is convergence in $L^2$ of the empirical measure (normalized by the mean number of particles) if and only if an associated additive martingale is bounded in $L^2$. 
This is a significant improvement over previous results, which were mainly restricted to $\lambda$-positive motions. We then investigate under which conditions this limit is strictly positive on the event of non-extinction and show that this occurs whenever, on the latter event, particles do not accumulate on the boundary of the state space. \mbox{In particular,} this property also yields the convergence of the real empirical measure (normalized by the number of particles).
 Moreover, building on previous results we prove that if the motion is $\lambda$-positive then these limits hold also almost surely whenever the Doob's $h$-transform of this motion admits a suitable Lyapunov functional.
Finally, we illustrate our results for a variety of different motions: ergodic motions without absorption, $\lambda$-positive systems either transient or with absorption, and also certain non $\lambda$-positive systems such as the Brownian motion with negative drift killed at $0$. The validity of a law of large numbers for the empirical measure in this last example was announced by Kesten without proof in \cite{kesten1978}. Our results allow us to give a partial proof to Kesten's claim.

\end{abstract} 

\section{Introduction}

\subsection{State of the art and motivation}

Since the late seventies there has been an intensive effort of research dedicated to proving laws of large numbers for branching processes. 
This encompasses the empirical measure of the process normalized in two different ways: either by the mean number of particles (mean-normalized empirical measure or MNEM) or by the total number of particles (empirical measure or EM). 
The original result proved by Kesten and Stigum in \cite{kesten1966} considers a multi-type Galton-Watson process with a finite number of types and establishes the almost sure convergence for both MNEM and EM. In \cite{athreya2}, Athreya later proved the $L^2$-convergence using a spectral decomposition.
However, the situation becomes much more delicate when dealing with an infinite number of types or with a branching Markov process instead, i.e. a system of particles evolving according to a Markovian motion which branch independently (the latter can be seen as a continuous-time Galton-Watson process in which the space of types is identified with the motion's state space).
In this context, almost all available results require the system in consideration to be $\lambda$-positive. For a general homogeneous Markov process with a well-defined Doob's $h$-transform, $\lambda$-positivity simply means that the transformed process is ergodic. 
For branching Markov processes, this is not a mere technical assumption but it is in fact equivalent to supposing that the motion of a certain spine describing the genealogy of the branching process (which is sometimes described in the literature as an immortal particle)
is positive recurrent. Moreover, in the case of a constant branching rate, the branching process will be $\lambda$-positive if and only if the underlying motion is.
This property allows to implement convenient changes of measure which, in turn, somehow reduce the problem of proving a law of large numbers for the branching process into that of proving one for a positive recurrent Markov motion instead.
Early results in this direction \mbox{were given in \cite{moy1967,asmussen1976},} while the conceptual proof in \cite{kurtz1997} which introduced the notion of spine for the branching process sprouted a renewal of results for a variety of more complex systems, including  \cite{englander2010,harris2014,harris2015}.
Finally, in the recent works \cite{ChenShio2007,chen2016}, the previous techniques combined with functional analytical methods were used to obtain convergence results in the context of branching symmetric Hunt processes, under the assumptions that the branching process is $\lambda$-positive and that its generator has a spectral gap.
We refer to \cite{englander2014spatial,englandersurvey} for a more detailed overview of all known results and recent developments on this matter. 
However, we emphasize that, except for \cite{watanabe1967} in which the underlying motion is a $0$-drift Brownian motion killed at $0$ for which an explicit spectral decomposition is available (which is a remarkable and exceptional feature of this system), or the result in \cite{englander2009law} for superprocesses, in all of the aforementioned results the hypothesis of $\lambda$-positivity is key.
  
As a result, for general branching processes which are not $\lambda$-positive much less is known. Indeed, in these cases the trajectories of the spine are not localized, a fact which leads to considerably more involved situations. 
Nonetheless, one still expects similar LLN's to hold for such systems. As a matter of fact, in \cite{kesten1978} Kesten announced such a result for branching Markov processes with a constant branching rate and which are driven by a Brownian motion with negative drift killed at $0$,
which is perhaps the canonical example of a system which is not $\lambda$-positive. Kesten claimed that a strong LLN holds provided that the associated branching process is supercritical (i.e. $r>\frac{c^2}{2}$ where $r$ is the branching rate and $-c$ is the negative drift) but omits the proof of this fact which, to the best of our knowledge, has remained undisclosed so far. 

In addition to proving convergence of the MNEM, it is also of crucial importance to establish the strict positivity of its limit on the event of non-extinction. 
Indeed, it is only when this occurs that one can guarantee that, on the event of non-extinction, the number of particles scales like its mean and the EM converges to the (quasi)-stationary distribution of the driving motion, \mbox{see \eqref{eq:convpexp} below.} Moreover, as a consequence of the convergence of the EM one obtains a simple and direct method to simulate this distribution which, in many cases, is not explicitly known. However,
showing that the limit of MNEM is strictly positive on the event of non-extinction is a subtle question (in fact, it is not always true, see Section \ref{sec:tou} for instance) and the literature addressing this matter is very limited when dealing with more involved situations with absorption and/or infinite state spaces, see \cite{harris2009} for instance or also \cite{harris2006v1} for results on similar (but not identical) martingale limits.

\subsection*{Contribution}
Our main interest is to derive a general method for proving Kesten-Stigum type results for supercritical branching Markov processes with a constant branching rate, which works beyond the $\lambda$-positivity setting and holds under minimal assumptions on the underlying motion. 

We first study the $L^2$-convergence of MNEM. 
We note that, in general, the (mean-normalized) number of particles inside a given set $B$ is not a martingale here (not even when $B$ is the entire state space if in the presence of absorption) so that it is not trivial at all to prove its convergence. Still, our approach is purely probabilistic and uses only simple spinal decomposition techniques. More precisely, our main tools are the ``many-to-one'' and ``many-to-two'' lemmas which allow us to efficiently control the particle correlations as time tends to infinity. We thus manage to show that if the underlying motion has a well-defined Doob's $h$-transform in $L^2$ whose distribution is (in a sense) of regular variation as $t$ tends to infinity, then MNEM converges in $L^2$ if and only if an associated additive martingale is bounded in $\L^2$ and, furthermore, its limit coincides with that of the additive martingale. We shall wait until Section \ref{sec:results} to explain in which sense we understand this regular variation, but it will be clear once we do that all $\lambda$-positive systems have this property. However, the crucial point is that it is also satisfied in systems which are \textit{not} $\lambda$-positive, which allows us to establish convergence results beyond the scope of the previous works mentioned above. 
 
Then, we focus on investigating conditions which guarantee that the MNEM limit is \textit{always} strictly positive on the event of non-extinction (notice that this condition is stronger than just plain non-degeneracy of the limit).
Our main result in this respect is that this condition is equivalent to the property that, conditionally on non-extinction of the process, particles \textit{never} accumulate on the boundary of the state space, a property we shall call \textit{strong supercriticality} of the system.
In fact, we will show a more general result which states that, whenever MNEM converges in $L^2$ (plus other mild assumptions on the underlying motion), strong supercriticality is equivalent to the moment generating operator of the branching process having exactly two fixed points, a fact which readily implies the strict positivity of the MNEM limit on the event of non-extinction and, hence, from where convergence of EM then follows. 
This notion of strong supercriticality is in a way related to the concept of \textit{strong local survival} studied in \cite{bertacchi2014} and references mentioned therein, although we are not aware of any previous connections made between this and the convergence of both MNEM and EM.

Afterwards, we study the almost sure convergence of MNEM and EM in the $\lambda$-positive scenario. Building on techniques introduced in \cite{asmussen1976} (and later generalized in \cite{englander2010,ChenShio2007,chen2016}), we show that there is almost sure convergence provided that the $h$-transform of the original process (which is positive recurrent) admits a suitable Lyapunov functional. This is, in many situations, a more convenient condition for almost sure convergence than the ones in recent previous works, namely \cite{englander2010,ChenShio2007,chen2016}. For instance, for unbounded state spaces our condition will typically be considerably weaker than those in \cite{ChenShio2007,chen2016}, whereas finding a good Lyapunov functional will often be more direct than checking the assumptions in \cite{englander2010}. Indeed, we will see in the examples of Section \ref{sec:examples} that, for systems with absorption, a good choice for  Lyapunov functional is given by the $h$ in Doob's transform, so that in many cases one already has a natural candidate. See Section \ref{sec:pr} for further discussion.   
To conclude, we illustrate our results for a variety of different motions: ergodic motions without absorption, $\lambda$-positive systems either transient or with absorption, and finally the charismatic case of the Brownian motion with negative drift killed at $0$ which, as discussed above, is not $\lambda$-positive.
Despite what Kesten's result of almost sure convergence in \cite{kesten1978} might suggest, we show that convergence in $L^2$ of MNEM does not take place in the entire supercritical region $r > \frac{c^2}{2}$, but only if $r > c^2$. As a consequence, we obtain also the convergence of the EM in probability on this smaller region, thus giving a partial proof to Kesten's claim. To the best of our knowledge, this constitutes the first available proof for a result of this kind outside the $\lambda$-positive setting \mbox{(for branching} Markov processes, that is, since for superprocesses we have \cite{englander2009law}).

The rest of the article is structured as follows.
In Section \ref{sec:results} we first describe the basic setup and notation, and then state our main results. 
A detailed comparison with previous results is given in Subsection \ref{sec:pr}.
Section \ref{sec:examples} introduces several examples of application of our results.
\mbox{In Section \ref{sec:mtf}} we recall the many-to-few lemma, which constitutes one of the main tools of our analysis, while
Sections \ref{sec:proof2}, \ref{sec:proof}, \ref{sec:theo3} and \ref{sec:rpos} are devoted to the proofs of the main results. Later in Section \ref{sec:examples2}, we verify our assumptions for the examples of Section \ref{sec:examples}. Finally, Section 10 includes some final remarks and discusses future work plans.

\section{Preliminaries and main results}\label{sec:results}

\subsection{Preliminaries}

Let $X=(X_t)_{t \geq 0}$ be a homogeneous Markov process with cadlag trajectories on some metric space $\overline{J}$. We will assume throughout that:
\begin{enumerate}
	\item [$\bullet$] $X$ is allowed to have \textit{absorbing states}, i.e. $x \in \overline{J}$ such that, whenever $X_t = x$ for some $t$, one has $X_{s}=x$ for all times $s>t$.
	\item [$\bullet$] The set $\partial_* \overline{J}$ of all absorbing states of $X$ belongs to $\B$, the Borel $\sigma$-algebra of $\overline{J}$.
	\item [$\bullet$] $J:=\overline{J}-\partial_* \overline{J}$ is locally compact and separable. 
\end{enumerate}  
Now, consider the following branching dynamics:
\begin{enumerate}
	\item [i.] The dynamics starts with a single particle, located initially at some $x \in \overline{J}$, whose position evolves randomly according to $\mathcal{L}$, the infinitesimal generator of $X$. 
	\item [ii.] This initial particle branches at rate $r > 0$, dying and being replaced at its current position by an independent random number of particles $m \in \N_0$.
	\item [iii.] Starting from their birth position, now each of these new particles independently mimics the same stochastic behavior of its parent.
	\item [iv.] If a particle has $0$ children, then it dies and moves to a graveyard state $\Delta$ forever.
\end{enumerate}
Given any time $t \geq 0$, for each particle $u$ present in the dynamics at time $t$ we write $u_t$ to indicate its position at time $t$. Also, we let $\overline{\chi}_t$ denote the collection of particles in the branching dynamics which are alive at time $t$, i.e. $u_t \notin \Delta$. We identify $\overline{\chi}_t$ with a finite measure $\chi_t$ on $(\overline{J},\B)$ by setting 
$$
\chi_t := \sum_{u \in \overline{\chi}_t} \delta_{u_t}.
$$ Furthermore, let $\overline{\xi}_t$ denote the collection of particles $u$ in $\overline{\chi}_t$ which have not been absorbed yet, i.e. such that $u_t \in J$, and define its induced measure $\xi_t$ on $(J,\B_J)$ as
$$
\xi_t = \sum_{u \in \overline{\xi}_t} \delta_{u_t}.
$$ Finally, we write $|\xi_t|:=\xi_t(J)$ for the total mass of $\xi_t$, i.e. the number of living particles at time $t$ which have not been absorbed yet, and define the \textit{empirical measure} $\nu_t$ as
$$
\nu_t := \frac{1}{|\xi_t|} \cdot \xi_t 
$$ with the convention that $\infty \cdot 0 = 0$, used whenever $|\xi_t|= 0$.

Throughout the rest of the article we will use the subscript $x$ in the notation, e.g. in $P_x$ or $\E_x$, to indicate that the process involved in the corresponding probability or \mbox{expectation starts at $x$.} Similarly, the superscript $x$, e.g. in $\xi_t^{(x)}$ or $X_t^{(x)}$, shall indicate that the corresponding process starts at $x$. In order to avoid overloading the notation, we shall write $\E_x(|\xi_t|)$ instead of $\E_x(|\xi^{(x)}_t|)$ and similar simplifications whenever possible.

\subsection{The Kesten-Stigum theorem in $L^2$ 
	} \label{sec:main1}

We shall now present and discuss our main results. 
Before we can do so, however, we must introduce some assumptions on our branching dynamics. Our initial assumptions on the underlying motion $X$ are the following. 



 
\begin{assump} \label{assumpG} $\,$
	\begin{enumerate}
		\item [A1.] There exists $\lambda \geq 0$ and a nonnegative $\B$-measurable function $h:\overline{J} \rightarrow \R_{\geq 0}$ such that:
		\begin{itemize}
			\item [i.] $h(x)=0$ if and only if $x \in \partial_* \overline{J}$.
			\item [ii.] For every $x \in J$ the process $M^{(x)}=(M^{(x)}_t)_{t \geq 0}$ given by the formula
			$$
			M_t^{(x)} := \frac{h(X_t^{(x)})}{h(x)}e^{\lambda t},
			$$ is a (mean-one) square-integrable martingale, i.e. $-\lambda$ is an eigenvalue of the generator $\mathcal{L}$ with associated (right) eigenfunction $h$ satisfying $\E_x(h^2(X_t)) < +\infty$ for all $t \geq 0$.
		\end{itemize}
		
		\item [A2.] There exists a nonempty class of subsets $\mathcal{C}_X \subseteq \B_J$ such that for each $x \in \overline{J}$ and $B \in \mathcal{C}_X$ one has the asymptotic formula
		\begin{equation}\label{A2}
		P_x( X_t \in B ) = h(x)p(t)e^{-\lambda t}(\nu(B) + s_B(x,t)),
		\end{equation} for all $t > 0$, where $\lambda$ and $h$ are those from (A1) and:
		\begin{itemize}
			\item [i.] $\nu$ is a (non-necessarily finite) measure on $(J,\B_J)$ satisfying:
			\begin{enumerate}
				\item [$\bullet$] $\nu(B) \in [0,+\infty)$ for all $B \in \mathcal{C}_X$,
				\item [$\bullet$] there exists at least one $B' \in \mathcal{C}_X$ such that $\nu(B')>0$.
			\end{enumerate}
			\item [ii.] $p(t)$ is a regularly varying function at infinity, i.e. a function $p: (0,+\infty) \rightarrow (0,+\infty)$ such that the limit
			$$
			\ell(a):=\lim_{t \rightarrow +\infty} \frac{p(at)}{p(t)} 
			$$ exists and is finite for all $a > 0$. 
			\item [iii.] $s_B(\cdot,t)$ converges to zero as $t \rightarrow +\infty$ uniformly over 
			$J_n:=\{x \in J : \frac{1}{n}\leq h(x) \leq n\}$ for each $n \in \N$.
			
			\item [iv.] There exist $t_0,\overline{s}_B > 0$ such that $\sup_{x \in \overline{J}}s_B(x,t) \leq \overline{s}_B$ for all $t > t_0$.
			\end{itemize} 
	\end{enumerate}
\end{assump}

\begin{remark} Let us observe that if for $x \in J$ we consider the martingale change of measure $\tilde{P}_x$ (also known as Doob's $h$-transform) given by 
	\begin{equation} \label{eq:mcof}
	\frac{d\tilde{P}_x}{dP_x}\bigg|_{\F_t} = M_t^{(x)},
	\end{equation} where $(\F_t)_{t \geq 0}$ denotes the filtration generated by $X^{(x)}$, and also define the measure $\mu$ on $(J,\B_J)$ via the formula
	\begin{equation} \label{eq:defmu}
	\frac{d\mu}{d\nu}=h,
	\end{equation}then 
	\begin{equation}
	\label{eq:fasymp}
	P_x( X_t \in B) = \E_x (\mathbbm{1}_{B}(X_t)) = h(x)e^{-\lambda t} \tilde{\E}_x\left( \frac{\mathbbm{1}_B}{h}(X_t)\right) = h(x)p(t)e^{-\lambda t}\left( \nu(B) + s_B(x,t)\right)
	\end{equation} where $\tilde{\E}_x$ denotes expectation with respect to the measure $\tilde{P}_x$ and $s_B$ is given by 
$$
s_B(x,t):= \frac{1}{p(t)}\tilde{\E}_x\left( \frac{\mathbbm{1}_B}{h}(X_t)\right) - \nu(B) = \frac{1}{p(t)}\tilde{\E}_x\left( \frac{\mathbbm{1}_B}{h}(X_t)\right) - \mu\left(\frac{\mathbbm{1}_B}{h}\right).
$$ Thus, (A2) can be reformulated as the assumption that there exists a regularly varying function $p$ at infinity such that the limit
\begin{equation}
\label{eq:plimit}
\lim_{t \rightarrow +\infty} \frac{1}{p(t)}\tilde{\E}_x\left( \frac{\mathbbm{1}_B}{h}(X_t)\right) 
\end{equation} exists for all $B \in \C_X$, is given by the non-trivial measure $\mu$ and, moreover, \mbox{is uniform over each $J_n$.} In particular, it follows that, for any $B \in \mathcal{C}_X$ such that $\nu(B) > 0$, the function 
$$
\ell^{(x)}_B(t):=\tilde{\E}_x\left( \frac{\mathbbm{1}_B}{h}(X_t)\right)
$$ is of regular variation at infinity. In conclusion, we may regard (A2) as requiring that, under $\tilde{P}_x$, the distribution of $X^{(x)}$ be (in some sense) of regular variation at infinity uniformly over each $J_n$.    
\end{remark}
	
It is worth mentioning that Assumptions \ref{assumpG} are satisfied in many different situations. Indeed, as we will see in Section \ref{sec:rpos1} below, $\lambda$-positive processes (see Section \ref{sec:rpos1} for their precise definition) are natural examples of systems verifying these conditions and this already covers a wide range of possibilities: not only are ergodic processes in this category, but also certain transient systems can be $\lambda$-positive as well as many examples of almost-surely absorbed motions having $\nu$ as their Yaglom limit. Moreover, Assumptions \ref{assumpG} are also satisfied by processes which are not $\lambda$-positive, see for instance \cite{watanabe1967,kesten1978} and Section \ref{sec:bbm} for the cases of a zero and negative drift, respectively, Brownian motion killed at $0$. We also refer to \cite{KP1} for other examples of L\'evy processes satisfying these assumptions. 

On the other hand, we point out that the class $\mathcal{C}_X$ was introduced in Assumptions \ref{assumpG} because, in general, one cannot expect the asymptotics in \eqref{A2} to be valid for \textit{every} $B \in \B_J$, see \mbox{Section \ref{sec:tou}}. Nevertheless, even if this is not the case one can still produce convergence results which hold for all $B$ in this smaller class $\mathcal{C}_X$. Finally, we observe that in all the examples of Section \ref{sec:examples} the measure $\nu$ actually corresponds to the left eigenmeasure of the generator $\mathcal{L}$ associated to $-\lambda$. However, this fact will not be used throughout our analysis. 

Now, since we are interested in understanding the evolution in $L^2$ of the branching dynamics $\xi$ in the supercritical case in which $|\xi_t|$ remains positive for all times $t > 0$ with positive probability, we must also make the following assumptions on $m$ and $r$.

\begin{assump}\label{assumpG0} We shall assume throughout that the pair $(m,r)$ satisfies:	\begin{itemize}
		\item [I1.] $m_2:= \E(m^2) < +\infty$ and $m_1:= \E(m) > 1$.
		\item [I2.] $r(m_1-1) > \lambda$, where $\lambda$ is the parameter from Assumptions \ref{assumpG}.
	\end{itemize}
\end{assump} It follows from Assumptions \ref{assumpG0} and Lemmas \ref{lema:mt1}-\ref{lema:mt2} below that all second moments $\E_x(|\xi_t|^2)$ are well-defined for any $x \in J$ and $t \geq 0$, and also $\E_x(|\xi_t|) \rightarrow +\infty$ as $t \rightarrow +\infty$ for all $x$ as long as there exists at least one $B \in \mathcal{C}_X$ with $\nu(B)> 0$, which is precisely the case that interests us. 

Now, our first result is concerned with the $L^2$-convergence of the so-called \textit{Malthusian martingale} associated to our branching dynamics, which we define below.

\begin{definition} For any $x \in J$ we define the \textit{Malthusian martingale} $D^{(x)}=(D_t^{(x)})_{t \geq 0}$ as
	$$
	D_t^{(x)} := \frac{1}{h(x)}\sum_{u \in \overline{\xi}^{(x)}_t} h(u_t) e^{-(r(m_1-1)-\lambda)t}.
	$$
\end{definition} 
It follows from the many-to-one lemma in Section \ref{sec:mtf} and (A1) that $D^{(x)}$ is indeed a martingale. Furthermore, (A1) implies in fact that $D^{(x)}$ is square-integrable. Being also nonnegative, we know that there exists an almost sure limit $D^{(x)}_\infty$. Our first result is then the following.

\begin{theorem}\label{theo:main2} For every $x \in J$ we have that  
	$$
	\lim_{t \rightarrow +\infty} \E_x(D_t^2) = (m_2-m_1) \int_0^\infty \E_x(M_s^2) re^{-r(m_1-1)s}ds=:\Phi_x,
	$$ so that $D^{(x)}$ converges in $L^2$ to $D^{(x)}_\infty$ if and only if $\Phi_x < +\infty$. In this case, we have that
	$$
	\E_x(D_\infty)=1 \hspace{2cm}\text{ and }\hspace{2cm}\E_x(D_\infty^2)=\Phi_x.
	$$	
\end{theorem} 

We should point out that the $\Phi_x < +\infty$ condition is not trivial under Assumptions \ref{assumpG}, so that $L^2$-convergence may not always hold, see e.g. Section \ref{sec:bbm}. However, the $L^2$-convergence of $D^{(x)}$ is crucial as it dictates the validity of a Kesten-Stigum type result in $L^2$ for our branching dynamics, as our next result shows.

\begin{theorem}\label{theo:main}
	If for $x \in J$ and $B,B' \in \mathcal{C}_X$ with $\nu(B') > 0$ we define  $W^{(x)}(B,B')=(W^{(x)}_t(B,B'))_{t \geq 0}$ by the formula
	$$
W^{(x)}_t(B,B') := \frac{\xi^{(x)}_t(B)}{\E_x(\xi_t(B'))}
$$ (which is well-defined for $t$ large enough since $\liminf_{t \rightarrow +\infty} \E_x(\xi_t(B')) > 0$ by Lemma \ref{lema:mt1} and \eqref{A2}), then the following holds:
\begin{enumerate}
	\item [i.] The sequence $W^{(x)}(B,B')$ satisfies 
	$$
	\lim_{t \rightarrow +\infty} \E_x(W^2_t(B,B'))= \left[\frac{\nu(B)}{\nu(B')}\right]^2 \Phi_x.
	$$ In particular, it is bounded in $L^2$ if and only if $\Phi_x < +\infty$. 
	\item [ii.] If $W^{(x)}(B,B')$ is bounded in $L^2$ then we have that as $t \rightarrow +\infty$
	\begin{equation} \label{eq:conv1}
	W^{(x)}_t(B,B') \overset{L^2}{\longrightarrow} \frac{\nu(B)}{\nu(B')} \cdot D_\infty^{(x)}.
    \end{equation}
    In particular, conditionally on the event $\{D^{(x)}_\infty > 0\}$, we have that as $t \rightarrow +\infty$ 
    \begin{equation} \label{eq:ks2}
    \nu^{(x)}_t(B,B'):=\frac{\xi^{(x)}_t(B)}{\xi^{(x)}_t(B')} \overset{P}{\longrightarrow} \frac{\nu(B)}{\nu(B')}.
    \end{equation}
 \end{enumerate} 
       
\end{theorem} 

We note that, due to the presence of absorption, $W^{(x)}(B,B')$ will not be a martingale in general, so that the existence of an $L^2$-limit when bounded is by no means a trivial statement. However, it follows from Theorem \ref{theo:main} that it behaves asymptotically in $L^2$ as $\frac{\nu(B)}{\nu(B')}\cdot D^{(x)}$ which is a martingale, so that boundedness should still be enough to guarantee convergence (at it is indeed the case).

\subsection{Strict positivity of $D^{(x)}_\infty$ on the event of non-extinction}\label{sec:strict}

Observe that for every $x \in J$ the event $\Lambda^{(x)}:=\{D^{(x)}_\infty > 0\}$ is contained in the event of non-extinction \mbox{$\Theta^{(x)}:=\{ |\xi_t| > 0 \text{ for all }t\}$.} Ideally, we would like both events to be almost surely equal, i.e. 
\begin{equation}\label{eq:eq}
P_x( \Lambda | \Theta )=1,
\end{equation} since, in that case, Theorem \ref{theo:main} would tell us that, almost surely on the event of non-extinction, for any $B \in \mathcal{C}_X$ with $\nu(B) > 0$ the number of particles $\xi^{(x)}_t(B)$ grows like its expectation which, using the results from Section \ref{sec:mtf}, can be explicitly computed and, furthermore, that these particles distribute themselves according to $\nu$. The following result intends to give conditions under which the equality in \eqref{eq:eq} is guaranteed. It is based on the study of the moment generating operator associated to the branching dynamics, which we define now.

\begin{definition} \label{def:defG}
	Let $\mathfrak{B}$ denote the class of measurable functions $g: (J,\B_J) \to ([0,1],\B_{[0,1]})$, where $\B_{[0,1]}$ denotes the Borel $\sigma$-algebra on $[0,1]$. We define the \textit{moment generating operator} $G_T: \mathfrak{B} \to \mathfrak{B}$ by the formula
	$$
	G(g)(x):=\E_x \left( \prod_{u \in \overline{\xi}_1} g(u_1)\right)
	$$ with the convention that $\prod_{u \in \emptyset} = 1$, used whenever $|\xi_1| = 0$. 
\end{definition}

It is immediate to see that $\mathbf{1}$, the function constantly equal to one on $J$, is a fixed \mbox{point of $G$,} i.e. $G(\mathbf{1})=\mathbf{1}$. Furthermore, by the branching property of the dynamics one has that the functions 
\begin{equation}\label{eq:defeta}
\eta(x):= P_x( \Theta^c) \hspace{1cm}\text{ and }\hspace{1cm}\sigma(x):=P_x(D_\infty=0)
\end{equation} are also fixed points of $G$, see Proposition \ref{prop:fp} below.\footnote{The fact that both $\eta$ and $\sigma$ are indeed measurable holds if one assumes that the process $X^{(x)}$ is also measurable as a function of $x$ in a suitable manner.} Since clearly $\eta \leq \sigma$ and we also have $\sigma \neq \mathbf{1}$ since $\E_x(D_\infty)=1$ for all $x \in J$ by Theorem \ref{theo:main}, if we show that $G$ has at most two fixed points then this would imply that $\eta \equiv \sigma$ and so \eqref{eq:eq} would follow at once. Unfortunately, it is not always the case that $G$ has only two fixed points, see Section \ref{sec:tou} for a counterexample. Hence, we must impose some additional conditions for this to occur. First, we make some further assumptions.

\begin{assump}\label{assumpG2} Throughout Subsection \ref{sec:strict} we will make the following additional assumptions:
	\begin{enumerate}
	\item [B1.] $\Phi_x < +\infty$ for all $x \in J$.
	\item [B2.] For any $B \in \B_J$ with $\nu(B)>0$ there exists $B^* \in \mathcal{C}_X$ such that $B^* \subseteq B$ and $\nu(B^*)>0$.
	\item [B3.] The conditioned evolution of $X$ is \textit{irreducible}, i.e. for any pair $x \neq x' \in J$ and $B \in \B_J$ there exists $n=n(x,x',B) \in \N$ such that 
	$$
	 P_{x'}(X_1 \in B)> 0 \Longrightarrow P_{x}(X_{n+1} \in B) > 0.
	$$
	\end{enumerate}
\end{assump}

Assumption (B1) is not really restrictive, as we wish to focus here only on the case in which there is convergence in $L^2$. On the other hand, we impose (B2) in order to obtain an appropriate control on the growth of $\xi_t(B)$, namely that for each $n \in \N$ and $x \in J$
$$
\lim_{t \rightarrow +\infty} \left[\inf_{y \in J_n}\E_x(\xi_t(B)) \right]=+\infty \hspace{1cm}\text{ and }\hspace{1cm}\lim_{t \rightarrow +\infty} \xi^{(x)}_t(B) = +\infty \hspace{0.2cm}\text{ on }\{D^{(x)}_\infty > 0\},
$$ which follows from (B1-B2) by Theorem \ref{theo:main} (see Section \ref{sec:theo3} for details). We will see in Section \ref{sec:examples2} that (B2) is typically very easy to check. Finally, the notion of irreducibility in (B3) is different than that of $\psi$-irreducibility featured in \cite{meyn1993} and weaker than the standard definition of irreducibility when $J$ is countable. Although not entirely standard, it is nevertheless the notion which appears naturally in our analysis and it is satisfied in all applications of interest, see Section \ref{sec:examples2}.

Next, we introduce the notion of strong supercriticality which plays a key role in what follows.


\begin{definition} \label{def:ss} We shall say that the branching dynamics $\xi^{(x)}=(\xi^{(x)}_t)_{t \geq 0}$ is \textit{strongly supercritical} if:
	\begin{enumerate}
		\item [i.] $\xi^{(x)}$ is supercritical, i.e. $P_x(\Theta) > 0$.
		\item [ii.] $\eta(x)=P_x(\Gamma)$, where $\Gamma^{(x)}$ is the event defined as 
		$$
		\Gamma^{(x)}:=\left\{ \lim_{t \rightarrow +\infty} \left[ \min_{u \in \overline{\xi}_t^{(x)}} \Phi_{u_t}\right]= +\infty\right\}
		$$ with the convention that $\min_{u \in \emptyset} \Phi_{u_t}:=+\infty$, used whenever $|\xi_t|=0$. 
    \end{enumerate} Note that, provided (i) holds, (ii) is equivalent to the condition $P_x(\Gamma|\Theta)=0$. 
\end{definition}

We will see that in all the examples of Section \ref{sec:examples} the mapping $x \mapsto \Phi_x$ is bounded over subsets of $J$ which are at a positive distance from $\partial_* \overline{J}$ and, on the other hand, that it tends to infinity as $x$ approaches $\partial_* \overline{J}$. Thus, one can interpret strong supercriticality as the condition which states that on the event of non-extinction particles never accumulate on the boundary of the state space, $\partial_* \overline{J}$.  
On the other hand, we will see in Section \ref{sec:theo3} that under Assumptions \ref{assumpG2} strong supercriticality is equivalent to having 
\begin{equation} \label{eq:sscom}
P_x\left( \limsup_{t \rightarrow +\infty} \xi_t(B) > 0\right)=P_x(\Theta) > 0
\end{equation} for every $x \in J$ and all $B \in \B$ with $\nu(B) > 0$, which is the analogue in our context of the notion of strong local survival studied in \cite{bertacchi2014} and other references therein. However, in general it will not be equivalent to the concept of (plain) local survival introduced in \cite{EngKyp2004}, which is said to take place whenever there exists a compact set $\mathcal{K} \subseteq J$ such that
		\begin{equation}\label{eq:ls}
		P_x\left( \limsup_{t \rightarrow +\infty} \xi_t(\mathcal{K}) > 0\right) > 0.
		\end{equation} See Section \ref{sec:tou} for further details.
		
Our next result states that strong supercriticality is a necessary and sufficient condition for $G$ to have exactly two fixed points whenever under Assumptions \ref{assumpG2}.

\begin{theorem}\label{theo:main3} If Assumptions \ref{assumpG2} also hold then the following statements are equivalent:
\begin{enumerate}
	\item [i.] $G$ has exactly two fixed points, $\eta$ and $\mathbf{1}$.
	\item [ii.] $\eta(x) = \sigma(x)$ for all $x \in J$.
	\item [iii.] $\eta(x) = \sigma(x)$ for some $x \in J$.
    \item [iv.] $\xi^{(x)}$ is strongly supercritical for some $x \in J$. 
    \item [v.] $\xi^{(x)}$ is strongly supercritical for all $x \in  J$. 
\end{enumerate} 
\end{theorem}

We note that strong supercriticality is not a trivial condition under our current assumptions, not even for the particular case of $\lambda$-positive systems to be considered in Section \ref{sec:rpos1} below. Indeed, Section \ref{sec:tou} shows an example of a $\lambda$-positive system which satisfies Assumptions \ref{assumpG}, \ref{assumpG0} and \ref{assumpG2} but is not strongly supercritical. In particular, we have in this example that the random variable $D^{(x)}_\infty$ is zero with positive probability on the event $\Theta^{(x)}$ of non-extinction. 
Nonetheless, whenever $\xi^{(x)}$ is strongly supercritical this is not the case and so one obtains the following corollary.

\begin{corollary}\label{cor:convnu} If Assumptions \ref{assumpG2} hold and $\xi^{(x)}$ is strongly supercritical then $D^{(x)}_\infty > 0$ on $\Theta^{(x)}$. 
In particular, for every $B,B' \in \mathcal{C}_X$ with $\nu(B')> 0$ we have that, conditionally on $\Theta^{(x)}$, as $t \rightarrow +\infty$
	\begin{equation*} 
	\nu_t^{(x)}(B,B') \overset{P}{\longrightarrow} \frac{\nu(B)}{\nu(B')}.
	\end{equation*} 
\end{corollary}

Still, strong supercriticality appears to be a hard condition to check directly, at least in principle. In Section \ref{sec:examples2} we  introduce via examples some general methods to establish strong supercriticality which apply to a wide range of systems.

\subsection{The case of $\lambda$-positive systems}\label{sec:rpos1}

Perhaps the simplest example of an underlyling motions satisfying Assumptions \ref{assumpG} is that of a $\lambda$-\textit{positive} process, which we formally introduce now. 

\begin{definition}\label{def:rpos1} Given some fixed $\lambda \geq 0$, we will say that $X$ is a $\lambda$-positive process if there exist \mbox{a measurable function} $h: J \rightarrow (0,+\infty)$ and a (not necessarily finite) measure $\nu$ on $(J,\B_J)$, both unique up to constant multiples, such that:
	\begin{itemize}
		\item [$\bullet$] For any $x \in J$ the process $M^{(x)}$ as defined in (A1) from Assumptions \ref{assumpG} is a martingale, i.e. $-\lambda$ is a right-eigenvalue of the generator $\mathcal{L}$ with associated eigenfunction $h$.
		\item [$\bullet$] For any bounded measurable $f: J \rightarrow \R$ and every $t \geq 0$ $$
		\int_J \E_x(f(X_t)) d\nu(x) = e^{-\lambda t} \int_J f(x) d\nu(x),
		$$ i.e. $-\lambda$ is a left-eigenvalue of the generator $\mathcal{L}$ with associated eigenmeasure $\nu$.
		\item [$\bullet$] The eigenvectors $h$ and $\nu$ are such that
		$$
		\nu(h):=\int_J h(x) d\nu(x) < +\infty
		$$ and, furthermore, if one chooses their normalization so that $\nu(h)=1$, then for all $x \in J$ one has that, under the change of measure $\tilde{P}_x$ defined in \eqref{eq:mcof}, the process $X^{(x)}$ is ergodic with invariant probability measure $\mu$ given by \eqref{eq:defmu}.
	\end{itemize}
\end{definition}

Observe that if $X$ is $\lambda$-positive then, since $X^{(x)}$ is ergodic under $\tilde{P}_x$ with invariant measure $\mu$, for any $B \in \B_J$ such that the function $\frac{\mathbbm{1}_B}{h}$ is bounded and $\mu$-almost surely continuous, we have 
$$
\lim_{t \rightarrow +\infty} \tilde{\E}_x\left( \frac{\mathbbm{1}_B}{h}(X_t)\right) = \mu\left( \frac{\mathbbm{1}_B}{h}\right)= \nu(B)
$$ as $t \rightarrow +\infty$, so that by \eqref{eq:fasymp} the asymptotic formula \eqref{A2} holds for any such $B$ by taking $p(t)\equiv 1$. 
Thus, we see that $\lambda$-positive process fall naturally into the context of Assumptions \ref{assumpG}. Nonetheless, checking the uniform convergence of $s_B$ over each $J_n$ or the \mbox{square-integrability of $M_t$} will require some further conditions to be imposed on the process. Indeed, in general $\lambda$-positivity alone will not be enough to guarantee the validity of a Kesten-Stigum type result, as one typically needs to impose additional conditions for this to occur. Different possibilities for such conditions were introduced in \cite{ChenShio2007,chen2016,englander2010}, which we shall analyze in detail later in Section \ref{sec:pr}. Here, we propose a new alternative condition to check the remainder of Assumptions \ref{assumpG} in the $\lambda$-positive setting and obtain Theorem \ref{theo:main} which is based on the existence of a Lypaunov functional for the process $X$ under the measure $\tilde{P}$.

\begin{definition}\label{def:lyapunov} A $\B_J$-measurable $V : J \rightarrow \R_{\geq 0}$ is called a \mbox{(\textit{geometric}) \textit{Lyapunov functional for }$X$} (under the measure $\tilde{P}$) whenever it satisfies: 
\begin{itemize}
	\item [V1.] There exists $t > 0$ such that for every $R > 0$ one can find $\alpha_R \in (0,1)$ verifying  
	$$
	|\tilde{\E}_x(f(X_t))-\tilde{\E}_y(f(X_t))| \leq 2(1-\alpha_R)\|f\|_\infty
	$$ for any bounded $\B_J$-measurable $f:J \rightarrow \R$ and all $x,y \in J$ such that $V(x)+V(y) \leq R$.
	\item [V2.] There exist constants $\gamma,K > 0$ such that for all $t \geq 0$ and $x \in J$ one has
	$$
	\tilde{\E}_x(V(X_t)) \leq e^{-\gamma t}V(x) + K.
	$$
\end{itemize} 
\end{definition}

One has the following result relating the validity of Theorem \ref{theo:main} for $\lambda$-positive processes to the existence of a $h$-locally bounded Lyapunov functional with a large enough growth at infinity. 

\begin{proposition}\label{prop:lyapunov1}
	If $X$ is $\lambda$-positive and admits a Lyapunov functional $V$ such that:
	\begin{enumerate}
		\item [V3.] $V$ is $h$-locally bounded, i.e. $\sup_{x \in J_n} V(x) <+\infty$ for each $n \in \N$,
		\item [V4.] $\left\|\frac{h}{1+V}\right\|_\infty <+\infty$,
\end{enumerate}
    then Assumptions \ref{assumpG} are satisfied for all $B \in \mathcal{C}_X$, where $\mathcal{C}_X$ here given by 
	\begin{equation}
	\label{eq:defcx}
	\mathcal{C}_X= \left\{ B \in \B_J :  \left\| \frac{\mathbbm{1}_B}{h}\right\|_\infty <+\infty \right\}.
	\end{equation} Furthermore, there exists a constant $C=C(r(m_1-1),\lambda,V) > 0$ such that for all $x \in J$
	\begin{equation} \label{eq:compvc}
	\Phi_x \leq C \cdot \frac{1+V(x)}{h(x)} < +\infty
	\end{equation} In particular, the convergences in \eqref{eq:conv1} and \eqref{eq:ks2} hold for any $x \in J$ and $B,B' \in \mathcal{C}_X$ with $\nu(B')> 0$, and assumptions (B1)-(B2) are also satisfied.
\end{proposition} 

As a matter of fact, whenever $X$ is $\lambda$-positive and admits such a Lyapunov functional, one can establish \eqref{eq:conv1} and \eqref{eq:ks2} in the almost sure sense. This is the content of our last result.

\begin{theorem}\label{theo:main5} Suppose that $X$ is a $\lambda$-positive process such that the eigenfunction $h$ is continuous and which admits a Lyapunov functional $V$ verifying (V3-V4). Then, for each $x \in J$ such that 
	\begin{equation}
	\label{eq:phibarra}
	\overline{\Phi}_x:=\frac{(m_2-m_1)r}{h(x)} \int_0^\infty \tilde{\E}_x\left(h(X_s)(1+V(X_s))\right)e^{-(r(m_1-1)-\lambda)s}ds < +\infty
	\end{equation} there exists a full $P$-measure set $\Omega^{(x)}$ satisfying that:
	\begin{enumerate}
		\item [i.] For any $\omega \in \Omega^{(x)}$ one has
	\begin{equation}\label{eq:convas}
	\lim_{t \rightarrow +\infty} \frac{\xi^{(x)}_t(B)(\omega)}{\E_x(\xi_t(B'))}= \frac{\nu(B)}{\nu(B')}\cdot D^{(x)}_\infty (\omega)
	\end{equation} for all pairs $B,B' \in \mathcal{C}_X$ satisfying $\nu(\partial B)=0$ and $\nu(B')> 0$, where $\mathcal{C}_X$ is given by \eqref{eq:defcx}. 
	\item [ii.] For any $\omega \in \Omega^{(x)} \cap \Lambda^{(x)}$ one has
	\begin{equation} \label{eq:convas2}
	\lim_{t \rightarrow +\infty} \frac{\xi^{(x)}_t(B)(\omega)}{\xi^{(x)}_t(B')(\omega)} = \frac{\nu(B)}{\nu(B')}
	\end{equation} for all pairs $B,B' \in \mathcal{C}_X$ with $\nu(\partial B) = \nu(\partial B')=0$ and $\nu(B')> 0$. 
	\end{enumerate} 
\end{theorem} 

Theorem \ref{theo:main5} is very much in the spirit of previous results in \cite{englander2010, ChenShio2007, chen2016} and references therein. However, the main difference with these works lies in that we have replaced their integrability and uniform convergence to equilibrium assumptions on the process with the alternative requirement that there exists a suitable Lyapunov functional for $X$. As we will see in the examples of \mbox{Section \ref{sec:examples}} below, this alternative condition is often less restrictive than those in the aforementioned references and can be in occasions easier to check as well. Finally, note that Theorem \ref{theo:main5} can be combined with Theorem \ref{theo:main3} to obtain \eqref{eq:convas2} for any $\omega \in \Omega^{(x)} \cap \Theta^{(x)}$ whenever $\xi^{(x)}$ is strongly supercritical. 

\subsection{Comparison with some previous results}\label{sec:pr}

To conclude the section, let us now discuss how our assumptions/results compare to those of previous works from the literature. For simplicity, we consider here only those references which, to the best of our knowledge, contain the most recent progresses on this matter, but refer the reader to \cite{englander2010, englander2014spatial} for a more detailed analysis including also older works with relevant results, namely \cite{asmussen1976,watanabe1967}, and also for results in the related context of superprocesses, e.g \cite{englander2009law, Chen2015}. 

We begin by pointing out that in all three references \cite{englander2010, ChenShio2007, chen2016} laws of large numbers are established for branching Markov systems with a state-dependent branching rate $r=(r(x))_{x \in J}$ under the assumption of $\lambda$-positivity (of the branching process). As noted earlier, this does not always imply that the the underlying motion is itself $\lambda$-positive, although it effectively does so in the case we consider here of a constant branching rate. Thus, the first important difference with \cite{englander2010, ChenShio2007, chen2016} is that our results are not restricted to this type of motions, see Section \ref{sec:bbm} for instance. In fact, as far as we know, our approach is the first to yield general results of this type which go beyond the $\lambda$-positive setting, with the exception of course of \cite{kesten1978} but whose results are stated without proof and \cite{englander2009law} which deals with superprocesses (and keeping in mind that we are only considering constant branching rates). However, it is not entirely clear how our method can be of use to deal with state-depending branching rates, so that this is perhaps (so far) one drawback of our approach when compared to previous works. 

As far as our results in the $\lambda$-positive setting go, the main differences with  \cite{englander2010, ChenShio2007, chen2016} are that:
\begin{enumerate}
	\item [i.] In \cite{ChenShio2007} a functional analytic method is used which forces the authors to restrict their setting only to symmetric systems, whereas our approach (and also the one in \cite{englander2010}) allows to treat non-symmetric systems as well. Furthermore, even in the symmetric setting, their strong spectral assumptions are not satisfied in many of our applications. More precisely, with the exception of the model in Section \ref{sec:tou}, none of the examples of Section \ref{sec:examples} satisfy the assumption that the eigenfunction $h$ is bounded. Moreover, in this exceptional example the branching rate $r$ does not satisfy the assumption of being in the Kato class $K_\infty(X)$, which (for this particular system) consists of rapidly decaying functions.
	\item [ii.] In \cite{chen2016} these assumptions are relaxed and the authors are also able to consider cases with an unbounded eigenfunction $h$ or with branching rates not necessarily belonging to $K_\infty(X)$. However, the authors still require the assumption of \textit{asymptotic intrinsic ultracontractivity}, i.e. that there exist constants $c_1,c_2 > 0$ such that for all $t$ sufficiently large
\begin{equation} \label{eq:aiu}
	\sup_{x,y \in J} \left| \tilde{f}(x,y,t) - 1 \right| \leq c_1 e^{-c_2 t}
\end{equation} where $\tilde{f}(x,\cdot,t)$ here denotes the $\mu$-density of $X^{(x)}_t$ under the change of measure \mbox{given by $\tilde{P}_x$.} 
This uniform convergence over all of $J$ is not satisfied in any of the examples of Section \ref{sec:examples} where $J$ is unbounded.
	 \item [iii.] Finally, in \cite{englander2010} the authors require only that the following two conditions hold:
	 \begin{itemize}
	 	\item [E1.] $\nu(h^2):= \int_J h^2(x)d\nu(x) < +\infty$.
	 	\item [E2.] There exists a family of domains $(J_t)_{t \geq 0} \subseteq \B_J$ with $J_t \nearrow J$ such that:
	 	\begin{enumerate}
	 		\item [a)] One has $P_x$-almost surely that $\overline{\xi}^{(x)}_t$ is contained in $J_t$ for all $t$ large, i.e. 
	 		$$
	 		P_x \left( \lim_{t \rightarrow +\infty} \xi_t( J_t^c) = 0\right) = 1.
	 		$$
	 		\item [b)] For every compact set $\mathcal{K} \subseteq J$ one has
	 		\begin{equation}\label{eq:ehkp1}
	 		\lim_{t \rightarrow +\infty} \left[\sup_{x \in J_t,y \in \mathcal{K}} \left| \tilde{f}(x,y,t) - 1 \right| = 0 \right],
	 		\end{equation} where $\tilde{f}(x,\cdot,t)$ denotes the $\mu$-density of $X^{(x)}_t$ under $\tilde{P}_x$ as before.   
	 	\end{enumerate}
 	 \end{itemize} On the one hand, the existence of a Lyapunov functional $V$ for $X$ satisfying $\left\|\frac{h}{1+V}\right\|<+\infty$ is already enough to guarantee that $\nu(h^2)<+\infty$ holds (see Proposition \ref{prop:Lyapunov} below) whereas, in the setting of branching diffusions considered in \cite{englander2010}, (E1) alone implies that $\Phi_x < +\infty$ for all $x \in J$ (see \cite[Lemma 17]{englander2010}). Therefore, we see that there are no significant differences with our assumptions as far as (E1) goes.  
  On the other hand, (E2) is a milder version of \eqref{eq:aiu} in which, essentially, the family $(J_t)_{t \geq 0}$ is quantifying the speed of growth of the branching dynamics $\xi$ and (E2-b) ensures that the process $\xi$ is somehow ``mixing faster than it is spreading''. 
  Thus, the main difference with our approach is that we have replaced (E2) with the existence of a Lyapunov functional for $X$ satisfying (V3-V4) as well as $\overline{\Phi}_x < +\infty$. In principle, this has two advantages: first, it does not require for us to quantify the speed of growth of $\xi$ which may sometimes be difficult to do and, also, it is particularly useful whenever $\tilde{f}$ cannot be computed explicitly. Indeed, as we will see in Sections \ref{sec:gw} and \ref{sec:cp}, in many cases it is still possible to find an appropriate $V$ even though \eqref{eq:ehkp1} may become almost impossible to check directly. However, we do require the additional integrability condition $\overline{\Phi}_x<+\infty$ which, in principle, is not guaranteed by (E1). 
\end{enumerate}

\section{Examples and applications} \label{sec:examples}

We now illustrate our results through a series of examples. We present first the case of generic ergodic motions, which fall in the category of $0$-positive processes, and then proceed on to study four different models with $\lambda$-positive motions for $\lambda > 0$. Finally, we conclude in Section \ref{sec:bbm} with our results for the Branching Brownian Motion with a negative drift and absorption at the origin, which is the canonical example of a system with an underlying motion which is not $\lambda$-positive and that constitutes our most significant and novel contribution. 
In each case, we restate our results for their specific context, delaying the verification of all required assumptions to Section \ref{sec:examples2}. 

\subsection{Ergodic motions}\label{sec:cwa}

Suppose that $X$ is a motion without absorbing states, i.e. $\overline{J}=J$, which is ergodic and has stationary probability distribution $\nu$ on $(J,\B_J)$. In this case, it is easy to check that $X$ is $0$-positive and verifies \eqref{A2} with $h \equiv 1$, $p \equiv 1$, $\lambda = 0$ for any $B \in \B_J$ such that 
$$
s_B(x,t):= P_x(X_t \in B) - \nu(B) \underset{t \rightarrow +\infty}{\longrightarrow} 0,
$$ so that (A2) is verified for $\mathcal{C}_X$ given by
\begin{equation}
\label{eq:c}
\mathcal{C}_X:=\left\{ B \in \B_J : \lim_{t \rightarrow +\infty} s_B(\cdot,t) = 0 \text{ uniformly over $J$}\right\},
\end{equation} since $J_n\equiv J$ for all $n \in \N$ by choice of $h$. Moreover, with this choice of $h$ and $\lambda$, (A1) is immediate and so is (B1), since a direct computation shows that for all $x \in J$
$$
\Phi_x = \frac{m_2-m_1}{m_1-1}< +\infty.
$$ Finally, since there is no absorption and $\sup_{x \in J} \Phi_x < +\infty$, $\xi$ is immediately strongly supercritical. Thus, we see that, in principle, almost all of our assumptions are trivially satisfied in this setting. However, \eqref{eq:c} is not entirely satisfactory since $\mathcal{C}_X$ will most likely be empty unless $J$ is compact. This problem arises because of the way in which we have defined the sets $J_n$ in Section \ref{sec:results}, which is more adapted to handle more complex situations in which $h$ is non-constant. Indeed, \mbox{since $h\equiv 1$} here, the sets $J_n$ immediately coincide with all of $J$ whereas in the case of transient motions or absorbed motions in which $\nu$ is the Yaglom limit, these $J_n$ are typically compact sets which agree with all of $J$ only in the limit as $n \rightarrow +\infty$ (one can verify this in any of the \mbox{subsequent examples).} Still, we chose to maintain this definition of $J_n$ (instead of just taking $(J_n)_{n \in \N}$ as any increasing sequence of compact sets converging to $J$, for instance) because then Assumptions \ref{assumpG} become much simpler and more transparent for the (perhaps more relevant) case of a non-constant $h$. Fortunately, it is the case that Assumptions \ref{assumpG} can be relaxed, see Section \ref{sec:ergod2} below for details, and as a consequence we can obtain Theorem \ref{theo:main} also for ergodic systems where $J$ is not compact, by working instead with the (in this case) larger class 
\begin{equation} \label{eq:defctilde}
\tilde{\mathcal{C}}_X:=\left\{ B \in \B_J :\lim_{t \rightarrow +\infty} s_B(\cdot,t) = 0 \text{ uniformly over compact sets of $J$} \right\}.
\end{equation} 
To be more precise, we can summarize our results for ergodic motions as follows. 

\begin{theorem}\label{theo:ergod} Suppose that $(m,r)$ verifies (I1) and (I2) with $\lambda=0$. Then, if $X$ is an irreducible (in the sense of (B3)) ergodic process with stationary distribution $\nu$, the following assertions hold:  
	\begin{enumerate}
		\item [i.] For every $x \in J$ the martingale $D^{(x)}$ converges almost surely and in $L^2$ to some $D^{(x)}_\infty \in L^2$ which is strictly positive almost surely and satisfies 
		$$
		\E_x(D_\infty)=1 \hspace{2cm}\text{ and }\hspace{2cm}\E_x(D_\infty^2)=\frac{m_2-m_1}{m_1-1}.
		$$ 
		\item [ii.] For every $x \in J$ and every pair $B,B' \in \tilde{\mathcal{C}}_X$ with $\nu(B') > 0$ one has 
		\begin{equation} \label{eq:conv3}
		W^{(x)}_t(B,B') \overset{L^2}{\longrightarrow} \frac{\nu(B)}{\nu(B')} \cdot D_\infty^{(x)}\hspace{2cm}\text{ and }\hspace{2cm}\nu^{(x)}_t(B,B') \overset{P}{\longrightarrow} \frac{\nu(B)}{\nu(B')}.
		\end{equation}
		\item [ii.] The moment generating operator $G$ has exactly two fixed points, $\mathbf{0}$ and $\mathbf{1}$.
		\item [iii.] If $X$ admits a Lyapunov functional (as defined in Definition \ref{def:lyapunov}) which is bounded over compact subsets of $J$ then \eqref{eq:conv3} holds for all pairs $B,B' \in \B_J$ with $\nu(B')>0$ and, moreover, for each $x \in J$ 
		there exists a full $P$-measure set $\Omega^{(x)}$ such that for any $\omega \in \Omega^{(x)}$ one has
		$$
		\lim_{t \rightarrow +\infty} W^{(x)}_t(B,B')= \frac{\nu(B)}{\nu(B')}\cdot D^{(x)}_\infty (\omega) \hspace{0.8cm} \text{ and } \hspace{0.8cm}
		\lim_{t \rightarrow +\infty} \nu^{(x)}_t(B,B')(\omega) = \frac{\nu(B)}{\nu(B')},
		$$ for all pairs $B,B' \in \B$ such that $\nu(\partial B) = \nu(\partial B') = 0$ and $\nu(B')>0$.
	\end{enumerate}
\end{theorem}

\subsection{Subcritical Galton-Watson process}\label{sec:gw} Let us now move on to consider $\lambda$-positive motions with $\lambda > 0$.
As a first example, let $X$ be a continuous-time Galton-Watson process, i.e. a process on the state space $\overline{J}:=\N_0$ with transition rates $q$ given for any $x \in \N_0$ and $y \in \N^*:=\{-1\} \cup \N_0$ by
$$
q(x,x+y):=x \rho(y), 
$$ where $\rho$ is some probability vector in $\N^*$ representing the offspring distribution of each individual in the branching process (minus $1$). We assume that $X$ is subcritical, i.e. that $-\lambda:= \sum_y y \rho(y) <0$, so that $X$ is almost surely absorbed at $0$. 
It is well-known, see e.g. \cite{seneta2006}, that in this case $X$ is $\lambda$-positive with associated eigenfunction $h$ given by $h(x) \propto x$ and that $\nu$ is a finite measure assigning positive mass to all $x \in \N$, although an explicit expression for $\nu$ is, in general, \mbox{not known.} Furthermore, if one assumes that
	 $\rho(-1) \in (0,1)$ then $X$ is in fact irreducible (in the classical sense). Our results for this model are the following. 
	 
	\begin{theorem} Let $X$ be a Galton-Watson process with an offspring distribution $\rho$ satisfying 
		\begin{enumerate}
			\item [$\bullet$] $-\lambda:= \sum_{y} y\rho(y)< 0$,
			\item [$\bullet$] $\rho(-1) \in (0,1)$.
		\end{enumerate}
		Then, for any pair $(m,r)$ satisfying Assumptions \ref{assumpG0}, the following assertions hold for each $x \in \N$:
		\begin{itemize}
			\item [i.]  The martingale $D^{(x)}$ is bounded in $L^2$ if and only if $\sigma_\rho^2<+\infty$, where 
			$$
			\sigma_\rho^2:= \sum_{y} y^2\rho(y).
			$$ In this case, $D^{(x)}$ converges almost surely and in $L^2$ to some random variable $D^{(x)}_\infty \in L^2$ which is strictly positive on the event $\Theta^{(x)}$ of non-extinction and, moreover, satisfies 
			$$
			\E_x(D_\infty)=1 \hspace{2cm}\text{ and }\hspace{2cm}\E_x(D_\infty^2)=\Phi_x.
			$$ 
			\item [ii.] For all pairs $B,B' \subseteq \N$ with $B' \neq \emptyset$, $W^{(x)}(B,B')$ is bounded in $L^2$ if and only if $\sigma_\rho^2 < +\infty$. In this case, we have 
			$$
			W^{(x)}_t(B,B') \overset{L^2}{\longrightarrow} \frac{\nu(B)}{\nu(B')} \cdot D^{(x)}_\infty
			$$ and, conditionally on the event of non-extinction, also that 
			$$
			\nu^{(x)}_t(B,B') \overset{P}{\longrightarrow} \frac{\nu(B)}{\nu(B')}.
			$$
			\item [iii.] $\xi^{(x)}$ is strongly supercritical. In particular, the associated moment generating operator $G$ has exactly two fixed points, $\eta$ and $\mathbf{1}$. 
			\item [iv.] If in addition we have $k_\rho:=\sum_y y^3\rho(y) <+\infty$ then
			there exists a full $P$-measure set $\Omega^{(x)}$ such that for all pairs $B,B' \subseteq \N$ with $B' \neq \emptyset$ one has 
			$$
			\lim_{t \rightarrow +\infty} W^{(x)}_t(B,B')= \frac{\nu(B)}{\nu(B')}\cdot D^{(x)}_\infty (\omega)
			$$ for any $\omega \in \Omega^{(x)}$ and also that
			$$
			\lim_{t \rightarrow +\infty} \nu^{(x)}_t(B,B')(\omega) = \frac{\nu(B)}{\nu(B')}
			$$ for any $\omega \in \Omega^{(x)} \cap \Theta^{(x)}$.
		\end{itemize}
	\end{theorem} 
	
\subsection{Subcritical contact process on $\Z^d$ (modulo translations)}\label{sec:cp}

Let $\mathcal{P}_{f}(\Z^d)$ denote the class of all finite subsets of $\Z^d$ and $Y=(Y_t)_{t \geq 0}$ be the contact process on $\Z^d$, i.e. the Markov process on $\mathcal{P}_{f}(\Z^d)$ with transition rates $q$ given, for any $\sigma \in \mathcal{P}_{f}(\Z^d)$ and $x \in \Z^d$, by 
$$
q(\sigma,\sigma \cup \{x\}) = \gamma |\{ y \in \sigma : |y-x|_1=1\}| \hspace{1cm} \text{ and }\hspace{1cm} q(\sigma,\sigma-\{x\}) = \mathbbm{1}_{\sigma}(x),
$$ where $\gamma > 0$ is a fixed constant called the \textit{infection rate}. Notice that $Y$ is translation invariant, i.e. for any $\sigma \in \mathcal{P}_f(\Z^d)$ and $x \in \Z^d$ one has
$$ 
Y^{(\sigma)} \sim Y^{(\sigma + \{x\})} + \{-x\},
$$ so that there are no finite measures $\nu \neq 0$ verifying \eqref{A2} for $Y$.  
This can be fixed if one considers the process modulo translations. Indeed, say that two non-empty sets $\sigma,\sigma' \in \mathcal{P}_f(\Z^d)$ are \textit{equivalent} if they are translations of each other. Let $J$ denote the quotient space obtained from this equivalence and, for any non-empty $\sigma \in \mathcal{P}_f(\Z^d)$, let $\langle \sigma \rangle$ denote its corresponding equivalence class in $J$. Also, set $\langle \emptyset \rangle := \emptyset$ and $\overline{J}:= J \cup \{ \emptyset\}.$ Then, for any $\zeta \in J$ we define $X^{(\zeta)}$ by taking $\sigma_\zeta \in \mathcal{P}_f(\Z^d)$ such that $\langle \sigma_\zeta \rangle = \zeta$ and setting 
$
X^{(\zeta)}_t:= \langle  Y^{(\sigma_\zeta)}_t \rangle$. We call $X$ the \textit{contact process on $\Z^d$ modulo translations}.

It is well-known, see \cite{BezGri1990}, that $J$ is an irreducible class for the process $X$ and that there exists $\gamma_c=\gamma_c(d) > 0$ such that the absorbing state $\emptyset$ is reached almost surely if and only if $\gamma \leq \gamma_c$. Moreover, it has been shown that for $\gamma < \gamma_c$ then $\lambda, h$ and $\nu$ as in Assumptions \ref{assumpG} indeed exist and that the process is in fact $\lambda$-positive, see \cite{ferrari1996,andjel2015}, although neither $h$ nor $\nu$ are explicitly known. What is known, however, is that $\nu$ is finite and it assigns positive mass to every $x \in J$, see \cite{ferrari1996}. Our results for this system are the following.

\begin{theorem} Let $X$ be a contact process on $\Z^d$ modulo translations with infection rate $\gamma < \gamma_c$. Then, for any pair $(m,r)$ satisfying Assumptions \ref{assumpG0}, the following assertions hold for each $\zeta \in J$:
	\begin{itemize}
		\item [i.]  The martingale $D^{(\zeta)}$ converges almost surely and in $L^2$ to some random variable $D^{(\zeta)}_\infty \in L^2$ which is strictly positive on the event $\Theta^{(\zeta)}$ of non-extinction and, moreover, satisfies 
		$$
		\E_\zeta(D_\infty)=1 \hspace{2cm}\text{ and }\hspace{2cm}\E_\zeta(D_\infty^2)=\Phi_\zeta <+\infty.
		$$ 
		\item [ii.] For all pairs $B,B' \subseteq J$ with $B' \neq \emptyset$ we have as $t \rightarrow +\infty$ 
		$$
		W^{(\zeta)}_t(B,B') \overset{L^2}{\longrightarrow} \frac{\nu(B)}{\nu(B')} \cdot D^{(\zeta)}_\infty
		$$ and, conditionally on the event of non-extinction, also that 
		$$
		\nu^{(\zeta)}_t(B,B') \overset{P}{\longrightarrow} \frac{\nu(B)}{\nu(B')}.
		$$
		\item [iii.] $\xi^{(\zeta)}$ is strongly supercritical. In particular, the associated moment generating operator $G$ has exactly two fixed points, $\eta$ and $\mathbf{1}$. 
		\item [iv.] There exists a full $P$-measure set $\Omega^{(\zeta)}$ such that for all pairs $B,B' \subseteq J$ \mbox{with $B' \neq \emptyset$ one has}
		$$
		\lim_{t \rightarrow +\infty} W^{(\zeta)}_t(B,B')= \frac{\nu(B)}{\nu(B')}\cdot D^{(\zeta)}_\infty (\omega)
		$$ for any $\omega \in \Omega^{(\zeta)}$ and also that
		$$
		\lim_{t \rightarrow +\infty} \nu^{(\zeta)}_t(B,B')(\omega) = \frac{\nu(B)}{\nu(B')}
		$$ for any $\omega \in \Omega^{(\zeta)} \cap \Theta^{(\zeta)}$.
	\end{itemize}
\end{theorem} 

\subsection{Recurrent Ornstein-Uhlenbeck process killed at $0$} \label{sec:rou}

Consider a $1$-dimensional recurrent Ornstein-Ulhenbeck process which is killed at $0$, i.e. the stopped process $X=(X_t)_{t \geq 0}$ on $\R_{\geq 0}$ defined as $X_t:=Y_{t \wedge \tau_0}$, where $\tau_0:=\inf\{t \geq 0 : Y_t = 0\}$ and $Y$ is given by SDE 
$$
dY_t = -\lambda Y_t dt + dB_t, 
$$ for $B$ a standard ($1$-dimensional) Brownian motion and $\lambda > 0$ a fixed parameter called the \textit{drift}. 
The generator of $X$ has as domain the set of $C^2$-functions vanishing at $0$ (due to the killing at $0$) and is defined for any such $f$ as 
$$
\mathcal{L}(f)(x):= \frac{1}{2}f''(x) - \lambda xf'(x).
$$ 
It is well-known, see \cite{lladser2000}, that $X$ is $\lambda$-positive with eigenfunction $h(x):= \sqrt{\frac{4\lambda}{\pi}} x$ (when $h,\nu$ are normalized so that $\nu$ is a probability measure and $\nu(h)=1$) and eigenmeasure $\nu$ having density 
$
f_X(x):= 2\lambda xe^{-\lambda x^2} \mathbbm{1}_{{(0,+\infty)}}(x)
$ with respect to the Lebesgue measure $l$ on $\R$. We state our results for this model below. 

\begin{theorem} Let $X$ be a recurrent Ornstein-Uhlenbeck process with drift $\lambda > 0$ and killed at $0$. Then, for any pair $(m,r)$ satisfying Assumptions \ref{assumpG0}, the following assertions hold for each $x > 0$:
	\begin{itemize}
		\item [i.]  The martingale $D^{(x)}$ converges almost surely and in $L^2$ to some random variable $D^{(x)}_\infty \in L^2$ which is strictly positive on the event $\Theta^{(x)}$ of non-extinction and, moreover, satisfies 
		$$
		\E_x(D_\infty)=1 \hspace{2cm}\text{ and }\hspace{2cm}\E_x(D_\infty^2)=\Phi_x.
		$$ 
		\item [ii.] For all pairs $B,B' \subseteq \B_{(0,+\infty)}$ with $l(B')>0$ we have as $t \rightarrow +\infty$
		$$
		W^{(x)}_t(B,B') \overset{L^2}{\longrightarrow} \frac{\nu(B)}{\nu(B')} \cdot D^{(x)}_\infty
		$$ and, conditionally on the event of non-extinction, also that 
		$$
		\nu^{(x)}_t(B,B') \overset{P}{\longrightarrow} \frac{\nu(B)}{\nu(B')}.
		$$
		\item [iii.] $\xi^{(x)}$ is strongly supercritical. In particular, the associated moment generating operator $G$ has exactly two fixed points, $\eta$ and $\mathbf{1}$. 
		\item [iv.] There exists a full $P$-measure set $\Omega^{(x)}$ such that for all pairs of subsets $B,B' \subseteq \B_{(0,+\infty)}$ bounded away from $0$ which satisfy $l(\partial B)=l(\partial B')=0$ and $l(B')> 0$ one has
		$$
		\lim_{t \rightarrow +\infty} W^{(x)}_t(B,B')= \frac{\nu(B)}{\nu(B')}\cdot D^{(x)}_\infty (\omega)
		$$ for any $\omega \in \Omega^{(x)}$ and also that 
		$$
		\lim_{t \rightarrow +\infty} \nu^{(x)}_t(B,B')(\omega) = \frac{\nu(B)}{\nu(B')}
		$$ for any $\omega \in \Omega^{(x)} \cap \Theta^{(x)}$.
	\end{itemize}
\end{theorem} 

\subsection{Transient Ornstein-Uhlenbeck process} \label{sec:tou}
This example was considered originally in \cite{englander2010}. Let $X$ be the process with generator $\mathcal{L}$ defined for any $f \in C^2(\R)$ as 
$$
\mathcal{L}(f)(x):= \frac{1}{2}\sigma^2 f''(x) +\lambda xf'(x),
$$ 
where $\lambda,\sigma^2 > 0$ are called the \textit{drift} and \textit{dispersion} coefficients, respectively. In this case, one can check that $X$ is $\lambda$-positive with $h(x) \propto \exp\{-\frac{\lambda}{\sigma^2}x^2\}$ and $\nu$ given by the Lebesgue measure on $\R$. Unlike previous examples, here $\nu$ is an infinite measure so that, in particular, the asymptotics in \eqref{A2} does not hold for every $B \in \B_{\R}$ (for example, it does not hold if $B=\R$ because in this case \eqref{A2} cannot decay exponentially) but, by the transience of $X$, it does hold for any $B$ which is bounded. Still, this is enough to yield the following results.

\begin{theorem}\label{theo:rou} Let $X$ be a transient Ornstein-Uhlenbeck process with coefficients $\lambda,\sigma^2 > 0$. Then, for any pair $(m,r)$ satisfying Assumptions \ref{assumpG0}, the following assertions hold for each $x \in \R$:
	\begin{itemize}
		\item [i.]  The martingale $D^{(x)}$ converges almost surely and in $L^2$ to some random variable $D^{(x)}_\infty \in L^2$ which satisfies 
		$$
		\E_x(D_\infty)=1 \hspace{2cm}\text{ and }\hspace{2cm}\E_x(D_\infty^2)=\Phi_x.
		$$ 
		\item [ii.] For all pairs $B,B' \subseteq \B_{\R}$ of bounded sets with $l(B')> 0$ we have as $t \rightarrow +\infty$
		$$
		W^{(x)}_t(B,B') \overset{L^2}{\longrightarrow} \frac{\nu(B)}{\nu(B')} \cdot D^{(x)}_\infty
		$$ and, conditionally on the event of non-extinction, also that 
		$$
		\nu^{(x)}_t(B,B') \overset{P}{\longrightarrow} \frac{\nu(B)}{\nu(B')}.
		$$
		\item [iii.] $\xi^{(x)}$ is \textbf{not} strongly supercritical. In fact, $\eta(x)=0 < \sigma(x) < 1 = \mathbf{1}(x)$ so that $\sigma \neq \eta,\mathbf{1}$.
		\item [iv.] There exists a full $P$-measure set $\Omega^{(x)}$ such that for all pairs $B,B' \subseteq \B_{\R}$ of bounded sets such that $l(\partial B)=l(\partial B')=0$ and $l(B')> 0$ one has $$
		\lim_{t \rightarrow +\infty} W^{(x)}_t(B,B')= \frac{\nu(B)}{\nu(B')}\cdot D^{(x)}_\infty (\omega) $$ for any $\omega \in \Omega^{(x)}$ and also that
		\begin{equation}\label{eq:touconv}
		\lim_{t \rightarrow +\infty} \nu^{(x)}_t(B,B')(\omega) = \frac{\nu(B)}{\nu(B')}
		\end{equation} for any $\omega \in \Omega^{(x)} \cap \Lambda^{(x)}$ (observe that, due to the lack of strong supercriticality of $\xi^{(x)}$ here, we obtain  \eqref{eq:touconv} only on the event $\Omega^{(x)} \cap \Lambda^{(x)}$, which now is strictly contained in $\Omega^{(x)} \cap \Theta^{(x)}$ and, in fact, has a smaller probability by (iii)).
	\end{itemize}
\end{theorem} 
Parts (i) and (iv) of Theorem \ref{theo:rou} can also be found in \cite{englander2010} (together with the $L^1$ convergence of $W^{(x)}(B,B')$), whereas (ii) and (iii) are new results. On the other hand, we note that it follows from (iv) that for any compact set $\mathcal{K} \subseteq \R$ and $x \in \R$ we have 
$$
P_x\left( \limsup_{t \rightarrow +\infty} \xi_t(\mathcal{K}) > 0\right) \geq P_x (\Lambda) > 0
$$ but, however, from (iii) that $\xi^{(x)}$ is not strongly supercritical. This confirms our statement about the notion of local survival introduced in \eqref{eq:ls} being, in general, weaker than strong supercriticality.

\subsection{Brownian motion with drift killed at the origin} \label{sec:bbm}

Finally, we conclude with an example of an underlying motion which is not $\lambda$-positive. Consider a Brownian motion with negative drift $-c < 0$ killed at the origin, i.e. the process $X$ given by the generator 
$$
\mathcal{L}(f)(x):=\frac{1}{2}f''(x) - c f'(x)
$$ defined for all $C^2$-functions $f$ vanishing at $0$. It is shown in \cite{polak2012} that $X$ satisfies Assumptions \ref{assumpG} for $\lambda:=\frac{c^2}{2}$ and $\mathcal{C}_X:=\B_{(0,+\infty)}$, the class of Borel subsets of $(0,+\infty)$, and with $h(x):= \frac{1}{\sqrt{2\pi \lambda^2}} xe^{cx}$, $p(t):=t^{-\frac{3}{2}}$ and $\nu$ given by the density $f_X(x):=2\lambda x e^{-cx}\mathbbm{1}_{(0,+\infty)}(x)$ with respect to Lebesgue. However, $X$ is not a $\lambda$-positive motion since one can easily verify that $\nu(h)=+\infty$.\footnote{Nor is it $\lambda'$-positive for any other $\lambda'\neq \frac{c^2}{2}$ because otherwise \eqref{A2} would not hold for $\lambda=\frac{c^2}{2}$ and all $B \in \B_{(0,+\infty)}$.} Nonetheless, our approach still applies and thus we can obtain the following results.

\begin{theorem} \label{theo:bbm} Let $X$ be a Brownian motion with a negative drift $-c < 0$ killed at the origin.
	Then, for any pair $(m,r)$ satisfying Assumptions \ref{assumpG0}, the following assertions hold for any $x> 0$: 
	\begin{itemize}
		\item [i.]  $D^{(x)}$ is bounded in $L^2$ if and only if $r(m_1-1) > c^2=2\lambda$. In this case, $D^{(x)}$ converges almost surely and in $L^2$ to some random variable $D^{(x)}_\infty \in L^2$ which is strictly positive on the event $\Theta^{(x)}$ of non-extinction and, moreover, satisfies 
		$$
		\E_x(D_\infty)=1 \hspace{2cm}\text{ and }\hspace{2cm}\E_x(D_\infty^2)=\Phi_x.
		$$ 
		\item [ii.] For all pairs $B,B' \in \B_{(0,+\infty)}$ with $l(B')> 0$, the sequence $W^{(x)}(B,B')$ is bounded in $L^2$ if and only if $r(m_1-1)> c^2=2\lambda$. In this case, we have 
		$$
		W^{(x)}_t(B,B') \overset{L^2}{\longrightarrow} \frac{\nu(B)}{\nu(B')} \cdot D^{(x)}_\infty
		$$ and, conditionally on the event of non-extinction, also that 
		$$
		\nu^{(x)}_t(B,B') \overset{P}{\longrightarrow} \frac{\nu(B)}{\nu(B')}.
		$$
		\item [iii.] $\xi^{(x)}$ is strongly supercritical. In particular, the associated moment generating operator $G$ has exactly two fixed points, $\eta$ and $\mathbf{1}$. 
	\end{itemize}
\end{theorem}

Notice that, unlike the previous examples, for this last model $L^2$-convergence does not hold in the entire supercritical region $r(m_1-1)> \lambda$ but only in the smaller subregion $r(m_1-1)>2\lambda$.

\section{The many-to-few lemmas}\label{sec:mtf}

An element which will prove to be crucial in the proof of Theorem \ref{theo:main} is the ability to compute the first and second moments of the process $|\xi|=(|\xi_t|)_{t \geq 0}$ in exact form. We do this with the help of the many-to-few lemmas we state below. For simplicity, we will state only a reduced version of the many-to-one and many-to-two lemmas, which are all we need. For the many-to-few lemma in its full generality (and its proof) we refer to \cite{harris2015}. 

First, we state the many-to-one lemma. It receives this name because it reduces expectations involving random sums over \textit{many} particles, i.e. over all those in $\xi_t$, to expectations involving only \text{one} particle. 

\begin{lemma}[Many-to-one Lemma] \label{lema:mt1} Given a nonnegative measurable function $f:(\overline{J},\B) \rightarrow \R_{\geq 0}$, for every $t \geq 0$ and $x \in J$ we have
	$$
	\E_x \left( \sum_{u \in \overline{\chi}_t} f( u_t ) \right) = e^{r(m_1 - 1)t}\E_x \left( f(X_t) \right).
    $$
\end{lemma}

Next, we state the many-to-two lemma, used to compute correlations between pairs of particles. Before we can do so, however, we must introduce the notion of $2$-spine for our branching dynamics. 

\begin{definition} Consider the following coupled evolution on $\overline{J}$:
	\begin{enumerate}
		\item [i.] The dynamics starts with 2 particles, both located initially  at some $x \in J$, whose positions evolve together randomly, i.e. describing the same random trajectory, according to $\mathcal{L}$. 
		\item [ii.] The particles wait for an independent random exponential time $E$ of parameter $(m_2-m_1)r$ and then split at their current position, each of them then evolving independently afterwards according to $\mathcal{L}$. 
	\end{enumerate} 
	Now, for $i=1,2$, let $X^{(i)}=(X^{(i)}_t)_{t \geq 0}$ be the process which indicates the position of the $i$-th particle. We call the pair $(X^{(1)},X^{(2)})$ a $2$-spine associated to the triple $(m,r,\mathcal{L})$ and $E$ its splitting time. 
	\end{definition}
	
The many-to-two lemma then goes as follows.

\begin{lemma}[Many-to-two Lemma] \label{lema:mt2} Given any pair of measurable functions $f,g:(\overline{J},\B) \rightarrow \R_{\geq 0}$, for every $t \geq 0$ and $x \in J$ we have
		$$
		\E_x \left( \sum_{u,v \in \overline{\chi}_t} f( u_t )g(v_t) \right) = e^{2r(m_1 - 1)t}\E_x \left( e^{[\text{Var}(m) + (m_1-1)^2]r (E \wedge t)}f(X^{(1)}_t)g(X^{(2)}_t)\right),
		$$ where $(X^{(1)},X^{(2)})$ is a $2$-spine associated to $(m,r,\mathcal{L})$ and $E$ denotes its splitting time.
	\end{lemma}

\section{Proof of Theorem \ref{theo:main2}}\label{sec:proof2}

We first compute $\E_x( D_t^2)$ for every $t \geq 0$ and $x \in J$. Note that, by the many-to-two lemma and the definition of $2$-spine, a straightforward computation (see the proof of Theorem \ref{theo:main} for details) yields that
\begin{align*}
\E_x(D_t^2) &= \frac{1}{h^2(x)}\E_x \left( \sum_{u,v \in \overline{\xi}_t} h(u_t)h(v_t)e^{-2(r(m_1-1)-\lambda)t}\right)\\
&= \frac{e^{2\lambda t}}{h^2(x)}\E_x \left( e^{[\text{Var}(m) + (m_1-1)^2]r (E \wedge t)}h(X^{(1)}_t)h(X^{(2)}_t)\right)\\
& = [1]_t + [2]_t
\end{align*} where
$$
[1]_t := \frac{e^{2\lambda t}}{h^2(x)}\E_x \left( e^{[\text{Var}(m) + (m_1-1)^2]r (E \wedge t)}h^2(X^{(1)}_t)\mathbbm{1}_{\{E > t\}}\right) = \E_x(M_t^2)e^{-r(m_1-1)t}
$$
and 
$$
[2]_t := (m_2-m_1)r\int_0^t \E_x ( M_s^2\, \E_{X_s}^2 (M_{t-s}) ) e^{-r(m_1-1)s}ds. 
$$ Now, by (A1) we have that $M$ is a mean-one martingale so that $\E_{X_s}(M_{t-s})=1$ for all $s \in [0,t]$. Thus, we obtain that
$$
[2]_t= (m_2-m_1)r\int_0^t \E_x (M_s^2) e^{-r(m_1-1)s}ds.
$$ Recalling the definition of $\Phi_x$, it is then clear that $[2]_t \rightarrow \Phi_x$ as $t \rightarrow +\infty$ and, on the other hand, that whenever $\Phi_x < +\infty$ we have that $\liminf_{t \rightarrow +\infty}[1]_t \rightarrow 0$, so that $\liminf_{t \rightarrow +\infty} \E_x(D_t^2)= \Phi_x$. But, since $\lim_{t \rightarrow +\infty} \E_x(D_t^2)$ always exists (although it can be $+\infty$, in principle) because $\left(D^{(x)}\right)^2$ is a submartingale, we conclude that
$$
\lim_{t \rightarrow +\infty} \E_x( D_t^2) = \Phi_x.
$$ Being $D^{(x)}$ a martingale, this implies that it converges in $L^2$ if and only if $\Phi_x<+\infty$ and that, in this case, one has $\E_x(D_\infty^2)=\Phi_x$. Moreover, since $\E_x(D_t)=1$ for all $t \geq 0$, it also follows that $\E_x(D_\infty)=1$ and so this concludes the proof. 

\section{Proof of Theorem \ref{theo:main}}\label{sec:proof}

This section contains the proof of Theorem \ref{theo:main}. We will split the proof into two parts:

\begin{enumerate}
	\item [I.] First, we will show that, given $B,B' \in \mathcal{C}_X$ with $\nu(B')> 0$, for any $x \in J$ one has
	$$
	\lim_{t \rightarrow +\infty} \E_x(W^2_t(B,B')) = \left[\frac{\nu(B)}{\nu(B')}\right]\Phi_x.
    $$ 
	\item [II.] Then, we use (I) to conclude the convergence in \eqref{eq:conv1} whenever $\Phi_x < +\infty$. In particular, the convergence $W_t^{(x)}(B',B') \overset{P}{\longrightarrow} D_\infty^{(x)}$ together with \eqref{eq:conv1} yields that for any $B \in \mathcal{C}_X$ 
	$$
	\nu_t^{(x)}(B,B') \overset{P}{\longrightarrow} \frac{\nu(B)}{\nu(B')}
	$$ as $t \rightarrow +\infty$, conditionally on the event $\{D_\infty^{(x)} > 0\}$. 
\end{enumerate}
We dedicate a separate subsection to each parts, but begin first with a section devoted to proving two auxiliary lemmas to be used throughout the proof.

\subsection{Preliminary lemmas}

The first lemma we shall require is the following.

\begin{lemma}\label{lema:A3} If assumption (A1) holds then for any $T > 0$ we have
	\begin{equation} \label{eq:A3-ii}
	\lim_{n \rightarrow +\infty} \left[\sup_{t \in [0,T]} \E_x \left( M_t^2 \mathbbm{1}_{\{X_t \notin J_n\}}\right)\right] = 0.
	\end{equation}
\end{lemma}

\begin{proof} Notice that for any $t \in [0,T]$ we have the bound
	$$
	\E_x \left( M_t^2 \mathbbm{1}_{\{X_t \notin J_n\}}\right) \leq \frac{e^{2\lambda T}}{h^2(x)} \cdot \frac{1}{n^2} + \E_x \left( \left(\sup_{s \in [0,T]} M_s^2\right) \mathbbm{1}_{\{ \sup_{s \in [0,T]} h(X_s) > n\}}\right)  
	$$ so that it will suffice to show that 
	\begin{equation} \label{eq:lemaconv2}
	\lim_{n \rightarrow +\infty} \E_x \left( \left(\sup_{s \in [0,T]} M_s^2\right) \mathbbm{1}_{\{ \sup_{s \in [0,T]} h(X_s) > n\}}\right)=0.
	\end{equation} But since $\sup_{s \in [0,T]} M^2_s$ is $P_x$-integrable by Doob's inequality and (A1), and we also have that 
	$$
	\lim_{n \rightarrow +\infty} \mathbbm{1}_{\{ \sup_{s \in [0,T]} h(X_s) > n\}} = \mathbbm{1}_{\{ \sup_{s \in [0,T]} M_s^2 = +\infty\}},
	$$ where the right-hand side is now $P_x$-almost surely null by the integrability of $\sup_{s \in [0,T]} M^2_s$, using the dominated convergence theorem we can conclude \eqref{eq:lemaconv2}. 
\end{proof}

The second lemma concerns the asymptotic behavior of the function $p$ in \eqref{A2}.

\begin{lemma}\label{lema:p} The function $p$ from assumption (A2) satisfies:
	\begin{itemize}
		\item [i.] $p$ has subexponential growth, i.e. $\lim_{t \rightarrow +\infty} \frac{ \log p(t)}{t} = 0$.
		\item [ii.] If we define the function $q(t_1,t_2) := \frac{p(t_2)}{p(t_1+t_2)}$ then for any $C > 0$ we have that
		$$
		\limsup_{t_2 \rightarrow +\infty}\left[\sup_{t_1 \in [0,Ct_2]} q(t_1,t_2)\right]=:q_C < +\infty \hspace{1cm}\text{ and }\hspace{1cm}
		\lim_{t_2 \rightarrow +\infty} \left[\sup_{t_1 \in [0,C]} |q(t_1,t_2)-1|\right] = 0.
		$$
	\end{itemize}
\end{lemma}

\begin{proof} It is well-known that if $p$ is a regularly varying function at infinity then there exits $\alpha \in \R$ such that $p(t)= t^{\alpha} L(t)$ for some \textit{slowly} varying function $L$, i.e. a function $L :(0,+\infty) \rightarrow (0,+\infty)$ such that for any $a > 0$
	\begin{equation}
	\label{eq:limitingl}
	\lim_{t \rightarrow +\infty} \frac{L(at)}{L(t)} = 1.
\end{equation} Since it is straightforward to check that the function $t^\alpha$ satisfies (i) and (ii) above, it suffices to verify these claims for any slowly varying function $L$.
	
	To see this, we notice that by Karamata's representation theorem for any such $L$ there exists some $T > 0$ such that for every $t > T$ one has 
	$$
	L(t) = \exp\left\{ g_1(t) + \int_T^t \frac{g_2(s)}{s}ds\right\}
	$$ where $g_1,g_2 : (T,+\infty) \rightarrow \R$ are two bounded measurable functions which respectively satisfy
	$$
	\lim_{t \rightarrow +\infty} g_1(t) = b \in \R \hspace{2cm}\text{ and }\hspace{2cm}\lim_{t \rightarrow +\infty} g_2(t)=0.
	$$ From this representation (i) immediately follows. On the other hand, since the \mbox{convergence in \eqref{eq:limitingl}} is uniform whenever $a$ is restricted to compact sets of $(0,+\infty)$, (ii) now follows immediately from the writing
	$$
	\frac{L(t_2)}{L(t_1+t_2)} = \frac{L(t_2)}{L\left(t_2\left(1 + \frac{t_1}{t_2}\right)\right)} 
	$$ and \eqref{eq:limitingl}, using that $1+\frac{t_1}{t_2} \in [1,1+C]$. This concludes the proof.
\end{proof}

\subsection{Part I}\label{sec:part1} Assume first that $\Phi_x <+\infty$ and let us show that then one has
\begin{equation}\label{eq:phi1}
\lim_{t \rightarrow +\infty} \E_x(W^2_t(B,B'))=\left[\frac{\nu(B)}{\nu(B')}\right]^2 \Phi_x.
\end{equation} To this end, take $t > 0$ and notice that 
\begin{equation}
\label{eq:exp}
\E_x(W_t^2(B,B')) = \frac{\E_{x}(\xi^2_t(B))}{\E_x^2(\xi_t(B'))}.
\end{equation} Let us compute the expectations in the right-hand side of \eqref{eq:exp} by using the many-to-few lemmas. On the one hand, by the many-to-one lemma we have that
	\begin{equation}
	\label{eq:mto1}
	\E_x(\xi_t(B')) = \E_x \left( \sum_{u \in \overline{\chi}_t} \mathbbm{1}_{\{u_t \in B'\}} \right) = e^{r(m_1-1)t}P_x( X_t \in B').
	\end{equation} On the other hand, the many-to-two lemma yields
	\begin{align*}
	\E_x(\xi^2_t(B)) &= \E_x \left( \sum_{u,v \in \overline{\xi}_t} \mathbbm{1}_{\{u_t \in B\}}\mathbbm{1}_{\{v_t \in B\}}\right) \\
	& = e^{2r(m_1-1)t}\E_x \left(  \mathbbm{1}_{\{X^{(1)}_t \in B\}}\mathbbm{1}_{\{X^{(2)}_t \in B\}}e^{[\text{Var}(m)+(m_1-1)^2]r(E\wedge t)}\right).
	\end{align*} By separating in cases depending on whether $E > t$ or not, we obtain
	$$
	\E_x(\xi^2_t(B)) = (1)_{t} + (2)_{t},
	$$ where
	$$
	(1)_{t} := e^{r(m_2-1)t} \E_x (\mathbbm{1}_{\{X_t \in B\}} \mathbbm{1}_{\{E > t\}}) 
	$$ and 
	$$
	(2)_{t}: = e^{2r(m_1-1)t}\E_x\left(  \mathbbm{1}_{\{X^{(1)}_t \in B\}}\mathbbm{1}_{\{X^{(2)}_t \in B\}}e^{[\text{Var}(m)+(m_1-1)^2]rE}\mathbbm{1}_{\{E \leq t\}}\right).
	$$ Now, using the independence of $E$ from the motion of the $2$-spine, the Markov property yields
	$$
	(1)_{t}= e^{r(m_1-1)t}P_x(X_t \in B)
	$$ and
	$$
	(2)_{t}= (m_2-m_1)re^{2r(m_1-1)t} \int_0^t P_x\left(X^{(1),s}_t \in B, X^{(2),s}_t \in B\right)e^{-r(m_1-1)s}ds,
	$$ where $X^{(1),s}$ and $X^{(2),s}$ are two coupled copies of the Markov process $X$ which coincide \mbox{until time $s$} and then evolve independently after $s$. If we condition on the position of these coupled processes at time $s$, then we obtain
	\begin{equation} \label{eq:2th}
	(2)_{t,h} = (m_2-m_1)re^{2r(m_1-1)t} \int_0^t \E_x( P^2_{X_s}(X_{t-s} \in B) )e^{-r(m_1-1)s}ds.
	\end{equation} Now, from \eqref{eq:mto1} and \eqref{A2} we conclude that
	$$
	[1]_{t}:= \frac{(1)_{t}}{\E^2_x(\xi_t(B'))} = \frac{P_x(X_t \in B)}{P_x(X_t \in B')} \cdot \frac{1}{\E_x(\xi_t(B'))} = \frac{\nu(B)+s_B(x,t)}{[\nu(B')+s_{B'}(x,t)]^2} \cdot \frac{1}{h(x)p(t)}e^{-(r(m_1-1)-\lambda)t}
	$$ which, by (i) in Lemma \ref{lema:p}, shows that if $t$ is taken sufficiently large then 
	\begin{equation}
	\label{eq:bound1}
	|[1]_t| \leq 2 \frac{\nu(B)}{\nu(B')} \frac{e^{-\frac{1}{2}(r(m_1-1)-\lambda)t}}{h(x)}.
	\end{equation} Similarly, one has that
	$$
	[2]_{t}:=\frac{(2)_{t}}{\E^2_x(\xi_t(B))}= \int_0^t \Psi_{x,t}(s) ds,
	$$
	where
	$$
	\Psi_{x,t}(s):=(m_2-m_1)r \frac{\E_x( P^2_{X_s}(X_{t-s} \in B))}{P^2_x( X_t \in B')}e^{-r(m_1-1)s}.
	$$ To treat the term $[2]_{t}$ we split the integral into three separate parts, i.e. for $\alpha \in (0,1)$ and $T > 0$ to be specified later we write
	$$
	[2]_{t} = [a]_{t} + [b]_{t} + [c]_t
	$$ where
	$$
	[a]_{t} := \int_{\alpha t}^t \Psi_{x,t}(s)ds \hspace{1cm} 
	[b]_{t} := \int_{T}^{\alpha t} \Psi_{x,t}(s)ds \hspace{1cm}[c]_t:=\int_0^T \Psi_{x,t}(s)ds  . 
	$$ The first term $[a]_{t}$ deals with the case in which $s \rightarrow t$ and the asymptotics in \eqref{A2} for $P_{y}(X_{t-s} \in B)$ may not hold. In this case, if $\alpha$ is taken close enough to $1$ then $[a]_{t}$ tends to zero as $t \rightarrow +\infty$. 
	Indeed, notice that 
	$$
	\E_x( P_{X_s}^2(X_{t-s} \in B) ) \leq \E_x(P_{X_s}( X_{t-s} \in B)) = P_x( X_t \in B)
    $$ by the Markov property, so that
    $$
    [a]_{t} \leq (m_2-m_1)r \frac{P_x(X_t \in B)}{P_x^2(X_t \in B')} \int_{\alpha t}^\infty  e^{-r(m_1-1)s} ds = \frac{m_2-m_1}{m_1-1}e^{(1-\alpha)r(m_1-1)t} [1]_{t}.
    $$
	By recalling \eqref{eq:bound1}, if $\alpha$ is chosen sufficiently close to $1$ and $t$ taken sufficiently large then 
	\begin{equation}\label{eq:bounda}
	|[a]_{t}| \leq e^{-\frac{1}{4}(r(m_1-1)-\lambda)t}.
	\end{equation} Similarly to $[a]_{t}$, the term $[b]_t$ can also be made arbitrarily small as $t \rightarrow +\infty$ if $T$ is large enough. Indeed, by \eqref{A2} we have that
	\begin{align}
	\Psi_{x,t,h}(s) &= (m_2-m_1)r  \left[\frac{q(s,t-s)}{\nu(B')+s_{B'}(x,t)}\right]^2 \E_x(M_s^2(\nu(B)+s_B(X_s,t-s))^2)e^{-r(m_1-1)s} \label{eq:decomp}
	\\
	& \leq (m_2-m_1)8rq^2_{\frac{\alpha}{1-\alpha}} \left(\frac{\nu(B)+\overline{s}_B}{\nu(B')}\right)^2 \E_x(M^2_s) e^{-r(m_1-1)s} \nonumber
	\end{align}
	if $s \leq \alpha t$ and $t$ is large enough so as to have
	\begin{itemize}
		\item [$\bullet$] $\sup_{y \in J,\,u\geq(1-\alpha)t} s_B(y,u) \leq \overline{s}_B$,
		\item [$\bullet$] $q(s,t-s) \leq 2q_{\frac{\alpha}{1-\alpha}}$ (which can be done by (ii) in Lemma \ref{lema:p} since $\frac{s}{t-s}\leq \frac{\alpha}{1-\alpha}$),
		\item [$\bullet$] $s_{B'}(x,t) \geq -\frac{\nu(B')}{2}$,
	\end{itemize} so that
	\begin{equation} \label{eq:bcot}
	|[b]_{t}| \leq 8q^2_{\frac{\alpha}{1-\alpha}} \left(\frac{\nu(B)+\overline{s}_B}{\nu(B')}\right)^2 \cdot (m_2-m_1) \int_T^\infty \E_x(M^2_s)re^{-r(m_1-1)s}ds.
	\end{equation} Since $\Phi_x < +\infty$, the right-hand side of \eqref{eq:bcot} can be made arbitrarily small if $T$ is chosen sufficiently large depending on $\alpha$. 
	
	Finally, let us treat the last term $[c]_{t}$. By (A2) and (ii) in Lemma \ref{lema:p}, for $s \leq T$ we may write
	$$
	\Psi_{x,t,h}(s) = \frac{(m_2-m_1)r}{\nu^2(B')} (1 + o_t(1)) \E_x(M_s^2(\nu(B)+s_B(X_s,t-s))^2)e^{-r(m_1-1)s}
	$$ where $o_t(1)$ (which depends on $x,t,s$ and $B'$) tends to zero uniformly in $s \leq T$ as $t \rightarrow +\infty$. Thus, we may decompose
	$$
	[c]_{t}=[c_1]_{t} + [c_1^*]_{t}
	$$ with 
	$$
	[c_1^*]_{t}= \frac{m_2-m_1}{\nu^2(B')} \int_0^T \E_x(M_s^2(\nu(B)+s_B(X_s,t-s))^2))re^{-r(m_1-1)s}ds
	$$ and 
	\begin{equation} \label{eq:ccot1}
	|[c_1]_{t}| \leq \frac{(\nu(B)+\overline{s}_B)^2}{\nu^2(B')} \Phi_x \left[\sup_{s \leq T}o_t(1)\right], 
	\end{equation} where the right-hand side of \eqref{eq:ccot1} tends to zero as $t \rightarrow +\infty$ since $\Phi_x < +\infty$. Finally, given $n \in \N$ we decompose $[c_1^*]_{t}$ by splitting the expectation inside into two depending on whether $X_s \in J_n$ or not. More precisely, we write 
	$$
	[c_1^*]_{t}=[c_2]_{t}+[c_2^*]_{t}
	$$ where 
	$$
	[c_2^*]_{t} = \frac{(m_2-m_1)}{\nu^2(B')}\int_0^T \E_x(M_s^2(\nu(B)+s_B(X_s,t-s))^2 \mathbbm{1}_{\{X_s \in J_n\}})re^{-r(m_1-1)s}ds
	$$ and 
	\begin{equation}
	\label{eq:ccot2}
	|[c_{2}]_{t}| \leq \frac{m_2-m_1}{m_1-1} \frac{(\nu(B)+\overline{s}_B)^2}{\nu^2(B')} \sup_{s \in [0,T]} \E_x(M_s^2\mathbbm{1}_{\{X_s \notin J_n\}}).
	\end{equation} Notice that the right-hand side of \eqref{eq:ccot2} is independent of $t$ and tends to zero as $n$ tends to infinity for any fixed $T > 0$ by Lemma \ref{lema:A3}. On the other hand, observe that by (A2) and Lemma \ref{lema:A3}
	$$
	\E_x(M_s^2(\nu(B)+s_B(X_s,t-s))^2) \mathbbm{1}_{\{X_s \in J_n\}}) = \nu^2(B) \E_x(M_s^2)(1+\overline{o}_t(1))
	$$ where the term $\overline{o}_t$ (which depends on $x,n,t,s$ and $B$) tends to zero uniformly in $s \leq T$ as $t \rightarrow +\infty$ since $\sup_{s \in [0,T]} \E_x(M_t^2) \leq 4\E_x(M_T^2)$ by Doob's inequality. By repeating the same argument that lead us to \eqref{eq:ccot1}, we conclude that 
	$$
	[c_{2}^*]_{t,h}=[c_{3}]_{t} + [c_{4}]
	$$ where 
	$$
	[c_4] = \left[\frac{\nu(B)}{\nu(B')}\right]^2 \cdot (m_2-m_1)\int_0^T \E_x(M_s^2)re^{-r(m_1-1)s}ds
	$$ and $|[c_3]_{t}|$ tends to zero as $t \rightarrow +\infty$. Thus, we find that if we write $\Gamma:=\{1,a,b,c_1,c_2,c_3\}$ then 
	\begin{equation}
	\label{eq:cotfinal}
	\left| W^{(x)}_{t}(B,B') - \left[\frac{\nu(B)}{\nu(B')}\right]^2\Phi_x\right| \leq \sum_{i \in \Gamma} |[i]_t| + \left[\frac{\nu(B)}{\nu(B')}\right]^2 \cdot (m_2-m_1)\int_T^\infty \E_x(M_s^2)re^{-r(m_1-1)s}ds.
	\end{equation} By taking $\alpha$ adequately close to $1$, $T$ large enough (depending on $\alpha$) and then $n$ sufficiently large (depending on $T$), the right-hand side of \eqref{eq:cotfinal} can be made arbitrarily small for all $t$ large enough and so \eqref{eq:phi1} follows. 
	
	Now, let us assume that $\Phi_x = +\infty$ and show that $\lim_{t \rightarrow +\infty} \E_x(W^2_t(B,B'))=+\infty$ in this case, proving \eqref{eq:phi1}. To see this, we notice that for any fixed $T > 0$ we have
	$$
	\E_x\left(W^2_t(B,B')\right) \geq [c]_{t,0} = \int_0^T \Psi_{x,t,0}(s)ds.
	$$ If $n$ is chosen large enough so that $\sup_{s \in [0,T]} \E_x( M_s^2 \mathbbm{1}_{\{X_s \notin J_n\}}) < 1$, then for all $t$ large enough to  guarantee that
	\begin{enumerate}
		\item [$\bullet$] $\inf_{s \in [0,T]} \left( \frac{q(s,t-s)}{\nu(B')+s_{B'}(x,t)}\right)^2 \geq \frac{1}{2\nu^2(B')}$
		\item [$\bullet$] $\inf_{s \in [0,T], y \in J_n} (\nu(B)+s_B(y,t-s))^2 \geq \frac{\nu^2(B)}{2},$
	\end{enumerate} by \eqref{eq:decomp} we obtain 
	\begin{align*}
	[c]_{t,0} &\geq \left[\frac{\nu(B)}{\nu(B')}\right]^2\frac{(m_2-m_1)}{4} \int_0^T \E_x(M_s^2 \mathbbm{1}_{\{X_s \in J_n\}})re^{-r(m_1-1)s}ds\\
	\\
	&\geq \left[\frac{\nu(B)}{\nu(B')}\right]^2\frac{(m_2-m_1)}{4} \int_0^T \E_x(M_s^2) re^{-r(m_1-1)s}ds - \left[\frac{\nu(B)}{\nu(B')}\right]^2\frac{(m_2-m_1)}{4(m_1-1)}.
	\end{align*}The right-hand side of this last inequality can be made arbitrarily large by taking $T$ big enough, due to the fact that $\Phi_x=+\infty$. In particular, this implies that 
	$$
	\lim_{t \rightarrow +\infty}\E_x\left(W^2_t(B,B')\right)=+\infty
	$$ and thus concludes the proof of Part I.
	
\subsection{Part II}

We now check that, whenever $\Phi_x < +\infty$, one has 
		$$
		W^{(x)}_t(B,B') \overset{L^2}{\longrightarrow} \frac{\nu(B)}{\nu(B')} \cdot D^{(x)}_\infty.
		$$ for every $B,B' \in \mathcal{C}_X$ with $\nu(B')> 0$.  
Notice that by Theorem \ref{theo:main2} it suffices to show that  
		\begin{equation} \label{eq:kseq}
		\lim_{t \rightarrow +\infty} \left\| W^{(x)}_t(B,B') - \frac{\nu(B)}{\nu(B')} \cdot D^{(x)}_t \right\|_{L^2} = 0.
		\end{equation} Now, observe that
		$$
		\left\| W^{(x)}_t(B,B') - \frac{\nu(B)}{\nu(B')} \cdot D^{(x)}_t \right\|_{L^2}^2 = \E_x(W^2_t(B,B')) - 2\frac{\nu(B)}{\nu(B')}\E_x(W_t(B,B')D_t) + \left[\frac{\nu(B)}{\nu(B')}\right]^2\E_x(D^2_t)
		$$ so that, by \eqref{eq:phi1} and Theorem \ref{theo:main2}, \eqref{eq:kseq} will follow if we show that
		$$
		\lim_{t \rightarrow +\infty} \E_x(W_t(B,B')D_t) = \frac{\nu(B)}{\nu(B')} \Phi_x.
		$$ But this can be done by proceeding exactly as in Part I. We omit the details.
		
		Finally, that 
		$$
		\nu^{(x)}_t(B,B') \overset{P}{\longrightarrow} \frac{\nu(B)}{\nu(B')}
		$$ conditionally on the event $\{D^{(x)}_\infty > 0\}$ follows from the fact that 
		\begin{equation}\label{eq:convpexp}
		\frac{\E_x(\xi_t(B'))}{\xi^{(x)}_t(B')}=\frac{1}{W^{(x)}_t(B',B')}\overset{P}{\longrightarrow}\frac{1}{D^{(x)}_\infty} \hspace{1cm}\text{ and }\hspace{1cm}\frac{\xi^{(x)}_t(B)}{\E_x(\xi_t(B'))}=W^{(x)}_t(B,B') \overset{P}{\longrightarrow} \frac{\nu(B)}{\nu(B')} \cdot D^{(x)}_\infty
		\end{equation} conditionally on the event $\{D^{(x)}_\infty > 0\}$, which in turn follows from \eqref{eq:conv1}. This concludes Part II and thus the proof of Theorem \ref{theo:main}.
		

\section{Proof of Theorem \ref{theo:main3}}\label{sec:theo3}

We also divide the proof of Theorem \ref{theo:main3}, now into four parts. First, we show that $\eta$ and $\sigma$ are indeed fixed points of $G$ and, using (B3), that $\eta$ and $\mathbf{1}$ cannot intersect other fixed points of $G$. Next, we show that (B1-B2) imply that our branching dynamics can be  dominated from below by a supercritical Galton-Watson process. Using this domination, we then show that the notion of strong supercriticality can be reformulated in terms of certain fixed points of the operator $G$. Finally, we use this alternative formulation to show both implications of Theorem \ref{theo:main3}.

\subsection{Part I}

We begin by checking first that both $\eta$ and $\sigma$ are fixed points of $G$.

\begin{proposition} \label{prop:fp}
	The functions $\eta$ and $\sigma$ are fixed points of $G$.	
\end{proposition}

\begin{proof}
	Observe that for any $t > 0$ and $x \in J$ we have the relation
	$$
	|\xi_{1+t}^{(x)}| = \sum_{u \in \overline{\xi}^{(x)}_1} |\xi_t^{(u_1)}|, 
	$$ which implies that, for any $t > 0$, $|\xi^{(x)}_{1+t}|$ equals zero if and only if $|\xi_t^{(u_1)}|$ is zero for every $u \in \overline{\xi}_1^{(x)}$. Thus, if we take $t \rightarrow +\infty$ then the former yields 
	\begin{equation}
	\label{eq:eta1}
	\mathbbm{1}_{\left(\Theta^{(x)}\right)^c} = \prod_{u \in \overline{\xi}^{(x)}_1}\mathbbm{1}_{\left(\Theta^{(u_1)}\right)^c}.
	\end{equation} By taking expectations $\E_x$ on the equality in \eqref{eq:eta1}, we obtain that $\eta(x)=G(\eta)(x)$. Furthermore, since this holds for any $x \in J$, we conclude that $\eta$ is a fixed point of $G$.
	
	Now, to see that $\sigma$ is a fixed point of $G$, we observe the analogous relation
	$$
	D_{1+t}^{(x)} = \frac{1}{h(x)}\sum_{u \in \overline{\xi}^{(x)}_1} h(u_1) e^{-(r(m_1-1)-\lambda)} D_t^{(u_1)}
	$$ which, upon taking the limit $t \rightarrow +\infty$, becomes
	$$
	D_{\infty}^{(x)} = \frac{1}{h(x)}\sum_{u \in \overline{\xi}^{(x)}_1} h(u_1) e^{-(r(m_1-1)-\lambda)} D_\infty^{(u_1)}.
	$$ Since $h(y)=0$ if and only if $y \in \partial_* \overline{J}$, that $\sigma$ is a fixed point of $G$ now follows as before.  
\end{proof}

Next, we use irreducibility to see that $\eta$ and $\mathbf{1}$ cannot intersect other fixed points of $G$.

\begin{proposition}\label{prop:fpG}
	Assume that (B3) holds. Then, if $g$ is a fixed point of $G$ we have that:
	\begin{enumerate}
		\item [i.] $\eta(x) \leq g(x) \leq 1$ for all $x \in J$.
		\item [ii.] $g(x) = \eta(x)$ for some $x \in J \Longrightarrow g \equiv \eta$.
		\item [iii.] $g(x) = 1$ for some $x \in J \Longrightarrow g \equiv \mathbf{1}$. 
	\end{enumerate}
\end{proposition} 

\begin{proof} We show first that if $g$ is a fixed point of $G$ then $\eta \leq g \leq \mathbf{1}$. Indeed, the $g \leq \mathbf{1}$ inequality is immediate whereas the $\eta \leq g$ inequality follows from the fact that $G$ is an increasing operator, i.e. $G(f_1) \leq G(f_2)$ if $f_1 \leq f_2$, together with the fact that $\eta = \lim_{n \rightarrow +\infty} G^{(n)}(\mathbf{0})$, where $G^{(n)}$ denotes the $n$-th composition of $G$ with itself and $\mathbf{0}$ is the function constantly equal to $0$.

Now, let us prove (ii). First, we observe that it is easy to check by induction that for any $n \in \N$ 
$$
G^{(n)}(g)(x)= \E_x \left( \prod_{u \in \overline{\xi}_n} g(u_n) \right).
$$ In particular, if $x \in J$ satisfies $P_x( X_n \in \{ y \in J : \eta(y) < g(y) \}) > 0 ) > 0$ for some $n \in \N$ then, by considering only evolutions of $\xi^{(x)}$ in which there is no branching until time $n$, it is clear that 
$$
P_x\left( \prod_{u \in \overline{\xi}_n} \eta(u_n) < \prod_{u \in \overline{\xi}_n} g(u_n)\right) > 0
$$ so that 
$$
\eta(x)=G^{(n)}(\eta)(x)=\E_x \left( \prod_{u \in \overline{\xi}_n} \eta(u_n)\right) < 
\E_x \left( \prod_{u \in \overline{\xi}_n} g(u_n)\right)=G^{(n)}(h)(x)=g(x).
$$ Therefore, if $\eta(x)=g(x)$ then we must have $P_x( X_n \in \{ y \in J : \eta(y) < g(y) \}) = 0$ for every $n \in \N$. By irreducibility (assumption (B3)) we then obtain that $P_{x'}( X_{1} \in \{ y \in J : \eta(y) < g(y) \}) = 0$ and, as a consequence, that $P_{x'}( \xi_{1}( \{ y \in J: \eta(y) < g(y) \}) > 0 ) = 0$ holds for every $x' \in J$ since the particle positions $(u_1)_{u \in \overline{\xi}^{(x')}_1}$ are all distributed as $X^{(x')}_1$. Since $\eta$ and $h$ are fixed points of $G$, this implies that $g(x') \leq \eta(x')$ for every $x' \in J$ which, together with (i) shown above, allows us to conclude that $\eta \equiv g$. The proof of (iii) is analogous.
\end{proof}

\subsection{Part II} The following step is to show that, under (B1-B2), one has a suitable lower bound on the growth of our dynamics. To this end, for each $n \in \N$ we define the set 
\begin{equation}
\label{eq:defjotatilde}
\tilde{J}_n:=\{ x \in J : \Phi_x \leq n\}
\end{equation} and then write $\hat{J}_n:=J_n \cap \tilde{J}_n$. Notice that the sequence $(\hat{J}_n)_{n \in \N}$ is increasing and, furthermore, that $\cup_{n \in \N} \hat{J}_n = J$ by (B1). Now, the precise meaning of lower bound on the growth of our dynamics is formulated in the following definition. 

\begin{definition} \label{def:bstar} We say that the \textit{lower bound condition} holds, and denote it in the \mbox{sequel by (LB),} if for any $n \in \N$ and $B \in \B_J$ with $\nu(B) > 0$ there exists a time $T_{n,B} $ and a random variable $L_{n,B}$ satisfying $\E(L_{n,B}) > 1$ such that for all $x \in \hat{J}_n$ and every $t > T_{n,B}$ one has
	$$
	L_{n,B} \preceq \xi^{(x)}_t(B)
	$$ where $\preceq$ denotes stochastic domination, i.e. for any bounded measurable and increasing $f: \R \to \R$ one has that
	$$
	\E(f(L_{n,B})) \leq \inf_{x \in \hat{J}_n} \E_x( f(\xi_t(B))).
	$$ 
\end{definition}

\begin{remark}\label{rem:crec} Note that, by (B2) and Lemma \ref{lema:mt1} below, for any \mbox{$B \in \B_J$ with $\nu(B) > 0$ and $x \in J_n$} we have that 
	\begin{align}
	\E_x(\xi_t(B)) &\geq \E_x(\xi_t(B^*)) = h(x) p(t) e^{(r(m_1-1)-\lambda)t}(\nu(B^*)+s_{B^*}(x,t)) \nonumber\\
	& \geq \frac{1}{n}p(t) e^{(r(m_1-1)-\lambda)t}\left(\nu(B^*)+ \inf_{y \in J_n} s_{B^*}(y,t)\right) \label{eq:assump2B1}
	\end{align} so that, by Lemma \ref{lema:p} and (A2), for all $t$ large enough depending on $B$ and $n$ we have that
	\begin{eqnarray}
	\label{eq:assump2B}
	\inf_{x \in \hat{J}_n} \E_x(\xi_t(B)) > 1.
	\end{eqnarray} Thus, condition (LB) is simply a stronger form of \eqref{eq:assump2B}, one in which we ask the entire distributions of the random variables $(\xi_t^{(x)}(B))_{x \in \hat{J}_n}$ to be uniformly supercritical rather than just their means. 
\end{remark}

The lower bound (LB) will be the main tool in the proof of Lemma \ref{lemma:cru} in Part III below, which is crucial for proving Theorem \ref{theo:main3}. Our next result states that (LB) holds under (B2).  

\begin{proposition}\label{prop:b1} Assumption (B2) implies condition (LB). 
\end{proposition}

\begin{proof}
	Let us fix $n \in \N$ and notice that, by the Cauchy-Schwarz inequality, we have for any $x \in J_n$, $B \in \B_{J}$ and $K,T \in \N$ that
	\begin{equation} \label{eq:holder}
	\E_{x}^2(\xi_T(B)\mathbbm{1}_{\xi_T(B)\geq K}) \leq \E_x(\xi^2_T(B))P_x(\xi_T(B)\geq K).
	\end{equation}On the other hand, if $\nu(B)>0$ then it follows from (B2), \eqref{eq:assump2B1} and Assumptions \ref{assumpG} that 
	$$
	\lim_{T \rightarrow +\infty} \left[ \inf_{y \in J_n} \E_y(\xi_T(B))\right] = +\infty.
	$$ Therefore, by \eqref{eq:holder} we conclude that if $T$ is sufficiently large (depending only on $K$, $n$ and $B$) then for all $x \in J_n$ we have
	$$
	P_x(\xi_T(B)\geq K) \geq \frac{[\E_x(\xi_T(B))-K]^2}{\E_x(\xi^2_T(B))} \geq \frac{1}{2} \cdot \frac{\E_x^2(\xi_T(B))}{\E_x(\xi_T^2(B))}.
	$$ Now, a careful inspection of the proof of Theorem \ref{theo:main} shows that there exists a constant $C_n > 0$ and a time $T_n > 0$ such that 
	$$
	\frac{\E_x(\xi^2_T(B))}{\E_x^2(\xi_T(B))} \leq C_n \Phi_x + 1
	$$ for all $x \in J_n$ and $T > T_n$. We stress that $C_n$ and $T_n$ do \textit{not} depend on $x \in J_n$, only on $n$ and $B$. Therefore, since $\sup_{y \in \tilde{J}_n} \Phi_y <+\infty$, we may take $K \in \N$ sufficiently large and $T \in \N$ accordingly so that 
	$$
	\inf_{x \in \hat{J}_n} P_x(\xi_T(B)\geq K ) \geq \frac{1}{K-1}.
	$$ It follows that $L_{n,B} \preceq \xi^{(x)}_T(B)$ for any such $T$ and all $x \in \hat{J}_n$, where $L_{n,B}$ has distribution given by 
	$$
	P(L_{n,B}=K) = \frac{1}{K-1} = 1 - P(L_{n,B}=0).
	$$ Since in this case $\E(L_{n,B})=\frac{K}{K-1}>1$, this concludes the proof.
\end{proof}

\subsection{Part III}

We continue by using Proposition \ref{prop:b1} to show that strong supercriticality can be reformulated \mbox{in terms of} certain fixed points of $G$. More precisely, we have the following result.

\begin{proposition} \label{prop:ss1} If Assumptions \ref{assumpG2} are satisfied then $\xi^{(x)}$ is strongly supercritical if and only if the following two conditions hold:
	\begin{enumerate}
		\item [i.] $\xi^{(x)}$ is supercritical, i.e. $P_x(\Theta) > 0$.
		\item [ii'.] There exists $n \in \N$ such that 
	$$
	P_x(\Theta) = P_x \left( \limsup_{k \rightarrow +\infty} \xi_k(\tilde{J}_n) > 0\right)
	$$ where $\tilde{J}_n$ is given by \eqref{eq:defjotatilde}.
	\end{enumerate}
\end{proposition}

\begin{remark}\label{rem:fpG}Observe that for $B \in \B$ the function $g_{B}$ defined as
	$$
	g_{B}(x):= P_x\left( \limsup_{n \rightarrow +\infty} \xi_{n}(B) = 0\right)
	$$ is a fixed point of $G$. Indeed, the proof of this statement is analogous to that of Proposition \ref{prop:fp}. Thus, Proposition \ref{prop:ss1} states that $\xi^{(x_0)}$ is strongly supercritical if and only if for some $n \in \N$
	$$
	g_{\tilde{J}_n}(x_0)=\eta(x_0)<1.
	$$ But then by Proposition \ref{prop:fpG} we conclude that the same statement must hold for all $x \in J$, so that $\xi^{(x)}$ is strongly supercritical for some $x \in J$ if and only if it is strongly supercritical for all $x \in J$.
\end{remark}

We now prove Proposition \ref{prop:ss1}. Let us notice that it suffices to show that (ii) in Definition \ref{def:ss} is equivalent to (ii') in the statement above. To see the (ii')$\Longrightarrow$(ii) implication, notice the inclusions
\begin{equation} \label{eq:traninc}
\left(\Theta^{(x)}\right)^c \subseteq \Gamma^{(x)} = \bigcap_{n \in \N} \left\{ \lim_{t \rightarrow +\infty} \xi^{(x)}_{t}(\tilde{J}_n) = 0 \right\} \subseteq \bigcap_{n \in \N} \left\{ \lim_{k \rightarrow +\infty} \xi^{(x)}_{k}(\tilde{J}_n) = 0 \right\}.
\end{equation} Now, for each $n \in \N$ let us write $A_{n}^{(x)}:=\left\{ \lim_{k \rightarrow +\infty} \xi^{(x)}_{k}(\tilde{J}_n) = 0 \right\}$. Notice that, since the sequence $(\tilde{J}_n)_{n \in \N}$ is increasing, we have that $(A^{(x)}_{n})_{n \in \N}$ is decreasing and therefore that
\begin{equation} \label{eq:traninc2}
P_x\left(\bigcap_{n \in \N} \left\{ \lim_{k \rightarrow +\infty} \xi_{k}(\tilde{J}_n) = 0 \right\}\right) = \lim_{n \rightarrow +\infty} P_x(A_{n}).
\end{equation} Therefore, if (ii') holds then it follows that $P_x(\Theta^c) = P_x( A_{n} )$ for some $n$ and, by \eqref{eq:traninc} and \eqref{eq:traninc2}, we conclude that (ii) holds. On the other hand, if (ii) holds then by \eqref{eq:traninc} we have 
\begin{equation} \label{eq:tran3}
P_x(\Theta^c)= \lim_{n \rightarrow +\infty} P_x \left( \lim_{t \rightarrow +\infty} \xi_{t}(\tilde{J}_n)= 0\right).
\end{equation} Thus, if we show that for all $n \in \N$ sufficiently large so that $x \in \hat{J}_n$ and $\nu(\tilde{J}_n) > 0$ we have  
\begin{equation} \label{eq:tranimpl}
\lim_{k \rightarrow +\infty} \xi^{(x)}_{k}(\tilde{J}_n) = 0 \Longrightarrow \lim_{t \rightarrow +\infty} \xi^{(x)}_{t}(\tilde{J}_{n+1}) = 0
\end{equation} then, from \eqref{eq:tran3} and the inclusion $\left(\Theta^{(x)}\right)^c \subseteq \{ \lim_{k \rightarrow +\infty} \xi^{(x)}_{k}(\tilde{J}_n) = 0\}$, by iterating \eqref{eq:tranimpl} we conclude that for $n$ sufficiently large 
$$
P_x(\Theta^c)\leq P_x\left( \lim_{k \rightarrow +\infty} \xi_{k}(\tilde{J}_n) = 0\right) \leq \lim_{m \rightarrow +\infty} P_x \left( \lim_{t \rightarrow +\infty} \xi_{t}(\tilde{J}_m)= 0\right) = P_x(\Theta^c),
$$ which immediately gives (ii'). Now, \eqref{eq:tranimpl} follows at once from the next lemma.

\begin{lemma} \label{lemma:cru} For any $m \in \N$ and $B \in \B_J$ such that $\hat{J}_m \neq \emptyset$ and $\nu(B)>0$ there exists $T \in \N$ satisfying that for any $x \in \hat{J}_m$ one has
	$$
	\limsup_{t \rightarrow +\infty} \xi^{(x)}_t(\tilde{J}_m) > 0 \Longrightarrow \lim_{k \rightarrow +\infty} \xi^{(x)}_{kT}(B) = +\infty.
	$$
\end{lemma}

\begin{proof} The idea is to couple the sequence $(\xi^{(x)}_{kT}(B))_{k \in \N}$ together with an i.i.d. sequence $(Z^{(n)})_{n \in \N}$ of supercritical single-type Galton-Watson branching processes such that if at least one $Z^{(n)}$ survives on the event $\{\limsup_{t \rightarrow +\infty} \xi^{(x)}_t(\tilde{J}_m) > 0\}$ then $\xi^{(x)}_{kT}(B)$ tends to infinity as $k \rightarrow +\infty$. 
	
We proceed as follows. First we notice that, since condition (LB) holds by Proposition \ref{prop:b1}, there exists a random variable $L_{m,B}$ with $\E(L_{m,B})>1$ and a time $T \in \N$ such that for all $y \in \hat{J}_m$ and $t \geq T$ we have
\begin{equation} \label{eq:domi}
L_{m,B} \preceq \xi^{(y)}_t(B).
\end{equation} Next, given a fixed $x \in \hat{J}_m$, we define the process $V^{(1)}:=(V^{(1)}_j)_{j \in \N}$ by the formula
$$
V^{(1)}_j:= \xi^{(x)}_{jT}(B)
$$ and observe that for each $j \in \N$ we have
$$
V^{(1)}_{j+1} \geq \sum_{u \in \overline{\xi}^{(x)}_{jT}(B)} \xi^{(u)}_T(B),
$$ where $\overline{\xi}^{(x)}_{jT}(B)$ denotes the subcollection of particles of $\overline{\xi}^{(x)}_{jT}$ which are located inside the subset $B$ and, for $u \in \overline{\xi}^{(x)}_{t}$, $\xi^{(u)}$ is the sub-dynamics of $\xi^{(x)}$ associated to the particle $u$ starting at time $t$. 
Since for each $u \in \overline{\xi}^{(x)}_{jT}(B)$ we have that $L_{B,m} \preceq \xi^{(u)}_T(B)$ by \eqref{eq:domi}, it follows that (by enlarging the current probability space if necessary) one can couple $V^{(1)}$ with a Galton-Watson process $Z^{(1)}:=(Z^{(1)}_j)_{j \in \N}$ with offspring distribution given by $L_{m,B}$, in such a way that $V^{(1)}_j \geq Z^{(1)}_j$ holds for all $j \in \N$. Therefore, if $Z^{(1)}$ survives then $Z^{(1)}_j$ must tend to infinity as $j \rightarrow +\infty$ and, consequently, so must $\xi^{(x)}_{jT}(B)$. In this case, we decouple $\xi^{(x)}$ from the remaining $Z^{(n)}$ (for $n \geq 2$) by taking these to be independent from $\xi^{(x)}$. If $Z^{(1)}$ dies out, however, we proceed as follows:
\begin{enumerate}
	\item [i.] Define $\tau^{(1)}:=\inf \{ j \in \N : Z^{(1)}_j = 0\}$. Notice that $Z^{(1)}$ dies out if and only if $\tau^{(1)}<+\infty$.
	\item [ii.] If $\xi^{(x)}_{sT}(\tilde{J}_m)=0$ for all $s \geq \tau^{(1)}$, then decouple the $(Z^{(n)})_{n \geq 2}$ from $\xi^{(x)}$ as before.
	\item [iii.] If $\xi^{(x)}_{sT}(\tilde{J}_m)>0$ for some (random) $s \geq \tau^{(1)}$, then choose some $y \in \overline{\xi}^{(x)}_{sT}(\tilde{J}_m)$ at random and define the process $V^{(2)}=(V^{(2)}_j)_{j \in \N}$ according to the formula
	$$
	V^{(2)}_{j}:= \xi^{(y)}_{(\lceil s\rceil+j)T}(B),
	$$ where $\lceil s \rceil$ here denotes the smallest integer greater than or equal to $s$. Let us observe that, by construction, there exists a (random) $k \in \N$ such that $V^{(2)}_j \leq \xi^{(x)}_{(k+j)T}(B)$ for all $j \in \N$. By a similar argument than the one carried out for $V^{(1)}$, it is possible to couple $V^{(2)}$ with a Galton-Watson process $Z^{(2)}$ which is independent of $Z^{(1)}$ but has the same distribution, in such a way that $V^{(2)}_j \geq Z^{(2)}_j$ for all $j \in \N$. If $Z^{(2)}$ survives then, by the considerations above, $\xi^{(x)}_{jT}$ must tend to infinity as $j \rightarrow +\infty$. If not, then one can repeat this procedure to obtain a branching process $Z^{(3)}$ and so on.	
\end{enumerate} Since every $Z^{(n)}$ has the same \textit{positive} probability of survival, it follows that at least one of them will survive on the event $\{\limsup_{t \rightarrow +\infty} \xi^{(x)}_t(\tilde{J}_m) > 0\}$, and so the result now follows.
\end{proof}

\subsection{Part IV}

We now conclude by showing all implications in the statement of Theorem \ref{theo:main3}. 

First, let us observe that the condition $\Phi_x < +\infty$ implies that $\sigma(x)<1$. Indeed, if $\Phi_x < +\infty$ then by Theorem \ref{theo:main2} we have $\E_x(D_\infty)=1$ so that  $\sigma(x)=P_x(D_\infty=0) < 1$ necessarily holds. Thus, if (B1) holds then $\sigma \neq \mathbf{1}$ and therefore (i) must imply (ii). 

\mbox{That (ii) $\Longrightarrow$ (iii)} is obvious, so we move on to (iii) $\Longrightarrow$ (iv). Take $x \in J$ such that $\eta(x)=\sigma(x)$. Note that, by the argument given above, $\sigma(x)<1$ so that if $\eta(x)=\sigma(x)$ then $\xi^{(x)}$ is supercritical. It remains to verify (ii') of Proposition \ref{prop:ss1}. But (B2) together with Theorem \ref{theo:main} and \eqref{eq:assump2B1} imply that for any $B \in \B_J$ with $\nu(B)> 0$
$$
\limsup_{n \rightarrow +\infty} P_{x} ( \xi_n(B) = 0 | \Lambda )  \leq \limsup_{n \rightarrow +\infty} P_{x} ( \xi_n(B^*) = 0 | \Lambda ) = 0, 
$$ from which a straightforward calculation yields that  
$$
P_x \left( \left\{ \limsup_{n \rightarrow +\infty} \xi_n(B) > 0 \right\} \cap \Lambda \right) = P_x(\Lambda).
$$ Therefore, since we also have the inequalities
$$
P_x \left( \left\{ \limsup_{n \rightarrow +\infty} \xi_n(B) > 0 \right\} \cap \Lambda \right) \leq P_x\left( \limsup_{n \rightarrow +\infty} \xi_n(B) > 0\right) \leq P_x(\Theta),
$$ if $\eta(x)=\sigma(x)$ then we have $P_x(\Theta)=P_x(\Lambda)$ and so (ii') follows. This shows that (iii) $\Longrightarrow$ (iv). 

The implication (iv) $\Longrightarrow$ (v) is also obvious, so it remains to show (v) $\Longrightarrow$ (i). For this purpose, we note that, if $g \neq \mathbf{1}$ is a fixed point of $G$, Proposition \ref{prop:fpG} yields that $g(y) < 1$ for \textit{all} $y \in J$. Thus, by (B2) one can find $B \in \mathcal{C}_X$ with $\nu(B)>0$ and $\varepsilon > 0$ such that $\sup_{y \in B} g(y)<1-\varepsilon$. Now, since $g$ is a fixed point of $G$, we have
$$
g(x) = \lim_{n \rightarrow +\infty} G^{(n)}(g)(x) = \lim_{n \rightarrow +\infty}\E_{x} \left( \prod_{u \in \overline{\xi}_n} g(u_n) \right) \leq \liminf_{n \rightarrow +\infty} \E_{x}\left( (1-\varepsilon)^{\xi_n(B)}\right)
$$ where for the last inequality we have used the fact that $g \leq 1$. Moreover, let us observe that 
\begin{equation}
\label{eq:sscota}
\E_{x}\left( (1-\varepsilon)^{\xi_n(B)}\right) \leq P_{x}( N_n = 0) + P_{x}( \Theta^c \cap \{ N_n > 0\}) + 
\E_{x}\left( (1-\varepsilon)^{\xi_n(B)}\mathbbm{1}_\Theta\right),
\end{equation} where, since $\{N_n^{(x)} > 0\} \searrow \Theta^{(x)}$ as $n \rightarrow +\infty$, we have that 
$$
\lim_{n \rightarrow +\infty} P_{x}(N_n = 0) = \eta(x) \hspace{1cm}\text{ and }\hspace{1cm}\lim_{n \rightarrow +\infty} P_{x}(\Theta^c \cap \{N_n > 0\}) = P_{x} ( \Theta^c \cap \Theta) = 0.
$$ Hence, if we could show that 
\begin{equation}
\label{eq:sscota2}
\liminf_{n \rightarrow +\infty} \E_x\left( (1-\varepsilon)^{\xi_n(B)}\mathbbm{1}_{\Theta}\right) = 0,
\end{equation} then we would immediately obtain that $g(x) \leq \eta(x)$. Together with the obvious reverse inequality, this would yield $\eta(x)=g(x)$ and hence, by Proposition \ref{prop:fpG}, that $\eta \equiv g$. Since we have that $\eta \neq \mathbf{1}$ by the strong supercriticality of $\xi^{(x)}$, (v) $\Longrightarrow$ (i) would follow at once. Thus, let us show \eqref{eq:sscota2}.
Observe that \eqref{eq:sscota2} immediately follows if we can show that for any $K > 0$ 
\begin{equation} \label{eq:sscota3}
\liminf_{n \rightarrow +\infty} P_x ( \xi_n(B) \leq K | \Theta) = 0.
\end{equation} Since $\xi^{(x)}$ is strongly supercritical, by (ii') of Proposition \ref{prop:ss1} we have that \eqref{eq:sscota3} is then equivalent to 
\begin{equation} \label{eq:sscota4}
\liminf_{n \rightarrow +\infty} P_x \left( \xi_n(B) \leq K \Bigg| \limsup_{j \rightarrow +\infty} \xi_{j}(\tilde{J}_k) > 0\right) = 0
\end{equation} if $k \in \N$ is taken sufficiently large.
But \eqref{eq:sscota4} is now a straightforward consequence of Lemma \ref{lemma:cru}, so that \eqref{eq:sscota2} follows. 

\section{Proof of Proposition \ref{prop:lyapunov1} and Theorem \ref{theo:main5}}\label{sec:rpos}

We now prove our results for $\lambda$-positive processes, namely Proposition \ref{prop:lyapunov1} and Theorem \ref{theo:main5}.

\subsection{Proof of Proposition \ref{prop:lyapunov1}}

This is essentially a consequence of the following result. 

\begin{proposition}\label{prop:Lyapunov} If $X$ is $\lambda$-positive and admits a Lyapunov functional $V$ as in Definition \ref{def:lyapunov}, 
	then $\mu(g) < +\infty$ for any $\B_J$-measurable $g:J \rightarrow \R$ such that $\|\frac{g}{1+V}\|_\infty < +\infty$ and, furthermore, there exist constants $C,\alpha > 0$ such that for all $x \in J$, $t \geq 0$ and any $g$ as above one has
	$$
	|\tilde{\E}_x(g(X_t)) - \mu(g)| \leq C (1+V(x)) e^{-\alpha t} \left\|\frac{g-\mu( g)}{1+V}\right\|_\infty.
	$$
	\end{proposition}

\begin{proof} This result is a careful combination of \cite[Theorem 3.6]{hairer2016} and \mbox{\cite[Theorem 4.3, Theorem 6.1]{meyn1993},} using the fact that since 
$$
\mu(g) \leq \left\| \frac{g}{1+V}\right\|_\infty \mu(1+V) = \left\| \frac{g}{1+V}\right\|_\infty (1+ \mu(V))
$$ it suffices to check that $\mu(V) < +\infty$ to see that $\mu(g) < +\infty$ for any function $g$ as in the statement. We omit the details.
\end{proof}

Observe that if $X$ admits a Lyapunov functional $V$ satisfying (V4) then, by Proposition \ref{prop:Lyapunov}, \mbox{if we set}
$$
\mathcal{C}_X:=\left\{ B \in \B_J : \left\|\frac{\mathbbm{1}_B}{h}\right\|_\infty <+\infty\right\}
$$ then for any $B \in \mathcal{C}_X$ and $x \in J$ we have that 
$$
|s_B(x,t)|=\left|\tilde{\E}_x\left(\frac{\mathbbm{1}_B}{h}(X_t)\right) - \mu\left(\frac{\mathbbm{1}_B}{h}\right)\right| \leq 2C(1+V(x))e^{-\alpha t}\left\|\frac{\mathbbm{1}_B}{h}\right\|_\infty
$$ so that if (V3) holds then (A2-iii) is automatically satisfied. Furthermore, we always have that
$$
s_B(x,t) \leq \tilde{\E}_x\left(\frac{\mathbbm{1}_B}{h}(X_t)\right) \leq \left\|\frac{\mathbbm{1}_B}{h}\right\|_\infty < +\infty
$$ for any such $B$, so that (A2-iv) is also satisfied. Since we have already seen that \eqref{A2} and (A2-i-ii) hold whenever $X$ is $\lambda$-positive, this shows that (A2) is satisfied. On the other hand, if (V4) holds then again by Proposition \ref{prop:Lyapunov} we have that for any $x \in J$ and $t \geq 0$ 
\begin{equation}\label{eq:cotam}
\E_x(M_t^2) = \frac{e^{\lambda t}}{h(x)} \tilde{\E}_x(h(X_t)) \leq \frac{e^{\lambda t}}{h(x)} \left( \mu(h) + 2C(1+V(x))\left( \left\|\frac{h}{1+V}\right\|_\infty + \mu(h) \right)\right) < +\infty
\end{equation} so that (A1) is also satisfied. Finally, since from \eqref{eq:cotam} we can obtain in fact the bound
$$
\E_x(M_t^2) \leq C_{h,\mu}\frac{e^{\lambda t}}{h(x)} (1+V(x))
$$ for some constant $C_{h,\mu} > 0$, a straightforward computation shows that
$$
\Phi_x \leq \frac{m_2-m_1}{r(m_1-1)-\lambda} C_{h,\mu} \cdot \frac{1+V(x)}{h(x)}.
$$ In particular, this yields that (B1) immediately holds. Furthermore, since for any given $B \in \B_J$ we have that $B \cap J_n \in \mathcal{C}_X$ for all $n$, we obtain also (B2) by taking $B^*=B \cap J_n$ for $n$ large enough so as to guarantee that $\nu(B\cap J_n) \geq \frac{\nu(B)}{2}$. This concludes the proof of Proposition \ref{prop:lyapunov1}. 

\subsection{Proof of Theorem \ref{theo:main5}}
For the proof we shall essentially follow the approach in \cite{asmussen1976,englander2010,ChenShio2007}, introducing a few  modifications to their methods to adapt to our different hypotheses. 

First, given any bounded and $\B_J$-measurable function $f: J \rightarrow \R$, for each $x \in J$ let us define the process $U^{(x),f}=(U^{(x),f}_t)_{t \geq 0}$ by the formula
\begin{equation}\label{eq:defu}
U^{(x),f}_t:= \frac{1}{h(x)}\sum_{u \in \overline{\xi}^{(x)}_t} h(u_t)f(u_t)e^{-(r(m_1-1)-\lambda)t}.
\end{equation} Note that if we take $f \equiv 1$ then we recover our Malthusian martingale, i.e $U^{(x),f}_t= D^{(x)}_t$ for all $t$. Our goal is to show that there exists a full $P$-measure set $\Omega^{(x)}$ such that for any $\omega \in \Omega^{(x)}$
\begin{equation}\label{eq:convu}
\lim_{t \rightarrow +\infty}U^{(x),f}_t (\omega) =  \nu(h f) \cdot D^{(x)}_\infty (\omega).
\end{equation} for all functions of the form $f=\frac{\mathbbm{1}_B}{h}$ with $B \in \mathcal{C}_X$ such that $\nu( \partial B)=0$. By Lemma \ref{lema:mt1}, this will automatically imply \eqref{eq:convas} since $p(t)\equiv 1$ in the $\lambda$-positive case.
The first step in the proof of \eqref{eq:convu} is to obtain the desired convergence along lattice times, i.e. along time sequences $(t_n)_{n \in \N} \subseteq \R_{> 0}$ of the form $t_n=n\delta$ for $\delta > 0$ arbitrarily small. This will be a consequence of the next lemma. For simplicity, in the sequel we suppress the dependence of $U^{(x),f}$ on $f$ from the notation unless it becomes strictly necessary.

\begin{lemma}\label{lema:18}
	For any $\delta > 0$ we have as $n \rightarrow +\infty$
	$$
	U^{(x)}_{2n\delta} - \E_x(U_{2n\delta}|\F_{n\delta}) \overset{as}{\longrightarrow} 0,
	$$ where $(\F_t)_{t \geq 0}$ is the filtration generated by $\xi^{(x)}$.
\end{lemma}

\begin{proof} By the Borel-Cantelli lemma and Markov's inequality, it will suffice to show that
	\begin{equation} \label{eq:bcsum}
	\sum_{n \in \N} \E_x\left(\left(U_{2n\delta} - \E_x(U_{2n\delta}|\F_{n\delta})\right)^2 \right) < +\infty.
	\end{equation} To this end, we observe that, by proceeding as in the proof of \cite[Lemma 18]{englander2010} for the case $p=2$, one has for any $t \geq 0$
	\begin{align*} 
	\E_x\left(\left(U_{2t} - \E_x(U_{2t}|\F_{t})\right)^2 \Big|\F_{t}\right) & = \left[\frac{e^{-(r(m_1-1)-\lambda)t}}{h(x)}\right]^2\sum_{u \in \overline{\xi}^{(x)}_t} h^2(u_t)\E_{u_t}( (U_t - \E_u(U_t))^2)\\
	& \leq \left[\frac{e^{-(r(m_1-1)-\lambda)t}}{h(x)}\right]^2 \sum_{u \in \overline{\xi}^{(x)}_t} h^2(u_t) \E_{u_t}(U^2_t)\\
	& \leq  \left[\frac{\|f\|_\infty e^{-(r(m_1-1)-\lambda)t}}{h(x)}\right]^2 \sum_{u \in \overline{\xi}^{(x)}_t} h^2(u_t) \E_{u_t}(D_t^2)
	\end{align*} so that by Lemma \ref{lema:mt1} we obtain
	$$
	\E_x\left(\left(U_{2t} - \E_x(U_{2t}|\F_{t})\right)^2\right) \leq \left[\frac{\|f\|_\infty e^{-(r(m_1-1)-\lambda)t}}{h(x)}\right]^2 e^{r(m_1-1)t} \E_x( h^2	(X_t)\E_{X_t}(D_t^2)).
	$$ Now, by the proof of Theorem \ref{theo:main2} we now that 
	\begin{equation} \label{eq:dcomp}
	\E_{X_t}(D_t^2)= \E_{X_t}(M_t^2)e^{-r(m_1-1)t} + (m_2-m_1)r\int_0^t \E_{X_t} (M_s^2) e^{-r(m_1-1)s}ds.
	\end{equation} so that, using the Markov property for $X$, we conclude that
	\begin{equation} \label{eq:ec}
	\E_x( h^2(X_t)\E_{X_t}(D_t^2)) = \E_x(h^2(X_{2t}))e^{-(r(m_1-1)-2\lambda)t} + (m_2-m_1)r\int_0^t \E_x( h^2(X_{t+s}))e^{-(r(m_1-1)-2\lambda)s}ds.
	\end{equation} Upon observing that the second term in the right-hand side of \eqref{eq:ec} can be rewritten as  
	\begin{align*}
	e^{(r(m_1-1)-2\lambda)t}\left[(m_2-m_1)r\int_t^{2t}\E_{x}(h^2(X_s))e^{-(r(m_1-1)-2\lambda)s}ds\right],
	\end{align*} a straightforward calculation now yields that $\E_x\left(\left(U_{2t} - \E_x(U_{2t}|\F_{t})\right)^2\right)$ can be bounded from above by 
	\begin{equation}\label{eq:bcsum2}
	\frac{\|f\|_\infty^2}{h(x)} \left( \tilde{\E}_x(h(X_{2t}))e^{-(r(m_1-1)-\lambda)2t}+ (m_2-m_1)r \int_t^{2t} \tilde{\E}_x(h(X_s))e^{-(r(m_1-1)-\lambda)s}ds\right).
	\end{equation} Now, since $\tilde{\E}_x(h(X_s)) \leq \left\|\frac{h}{1+V}\right\|_\infty \tilde{\E}_x(1+V(X_s)) \leq \left\|\frac{h}{1+V}\right\|_\infty(1+V(x)+K)$ for any $s$ by (V2-V4), we see that 
	$$
	\E_x\left(\left(U_{2t} - \E_x(U_{2t}|\F_{t})\right)^2\right) \leq \frac{\|f\|_\infty^2}{h(x)} \cdot \left\|\frac{h}{1+V}\right\|_\infty(1+V(x)+K) \cdot \left(1+ \frac{(m_2-m_1)r}{r(m_1-1)-\lambda}\right) \cdot e^{-(r(m_1-1)-\lambda)t}
	$$ which is enough to imply \eqref{eq:bcsum}.
\end{proof}

Now, the convergence along lattice times will follow if we can show that, as $n$ tends to infinity, for any $\delta > 0$ one has
$$
\E_x(U_{2n\delta}|\F_{n\delta}) \overset{as}{\longrightarrow} \nu(h f) \cdot D^{(x)}_\infty
$$ or, equivalently, that 
$$
\E_x(U_{2n\delta}|\F_{n\delta}) - \nu(h f) \cdot D^{(x)}_{n\delta} \overset{as}{\longrightarrow} 0.
$$ By Lemma \ref{lema:mt1}, the former reduces to showing that
\begin{equation}\label{eq:convcs}
\Sigma^{(x)}_{n\delta} :=\frac{1}{h(x)}\sum_{u \in \overline{\xi}^{(x)}_{n\delta}} e^{-(r(m_1-1)-\lambda)n\delta}h(u_{n\delta})\left[ \tilde{\E}_{u_{n\delta}}(f(X_{n\delta})) - \nu(h f)\right] \overset{as}{\longrightarrow} 0.
\end{equation} Using Markov's inequality, we have by the Borel-Cantelli lemma that \eqref{eq:convcs} follows if we show that
\begin{equation} \label{eq:convcsl2}
\sum_{n \in \N} \E_x( \Sigma_{n\delta}^2) < +\infty.
\end{equation} But notice that, by proceeding as in the computation leading to \eqref{eq:dcomp}, we have 
$$
\E_x(\Sigma^2_{n\delta}) = \langle1\rangle_{n\delta} + \langle2\rangle_{n\delta}
$$ where 
$$
\langle 1\rangle_{n\delta}:= \E_x(M_{n\delta}^2 \cdot s_{h f}^2(X_{n\delta},n\delta) ) e^{-r(m_1-1)n\delta}
$$ and
$$
\langle 2 \rangle_{n\delta}:= (m_2-m_1) \int_0^{n\delta} \E_x\left( M_s^2 \cdot  \E^2_{X_s}( M_{n\delta-s} \cdot s_{h f}(X_{n\delta-s},n\delta))\right)re^{-r(m_1-1)s}ds
$$ where, for $y \in J$ and $s \geq 0$, we write
$$
s_{h f}(y,s):=\tilde{\E}_y(f(X_s))-\nu(h f).
$$ Now, since $\| \frac{f}{1+V} \|_\infty \leq \|f\|_\infty < +\infty$ by assumption on $f$, Proposition \ref{prop:Lyapunov} yields the existence of some constants $C,\alpha > 0$ such that for all $y \in J$ and $s \geq 0$
$$
|s_{h f}(y,s)| \leq C(1+V(y))e^{-\alpha s}\|f\|_\infty.
$$ On the other hand, since $f$ is bounded we always have that $\|s_{h f}\|_\infty \leq 2\|f\|_\infty$. Hence, we obtain
\begin{align*}
\langle 1 \rangle_{n \delta} &\leq \frac{4\|f\|^2_\infty}{h(x)} \tilde{\E}_x( h(X_{n\delta})) e^{-(r(m_1-1)-\lambda)n\delta} \\
& \leq \frac{4\|f\|^2_\infty}{h(x)} \left\|\frac{h}{1+V}\right\|_\infty \tilde{\E}_x\left(1+V(X_{n\delta})\right)e^{-(r(m_1-1)-\lambda)n\delta}\\
& \leq \frac{4\|f\|^2_\infty}{h(x)}\left\|\frac{h}{1+V}\right\|_\infty  (1+V(x)+K)  e^{-(r(m_1-1)-\lambda)n\delta} 
\end{align*} so that $\sum_{n \in \N} \langle 1 \rangle_{n \delta} < +\infty$ and, similarly, using (V2) we obtain 
\begin{align*}
\E_x\left( M_s^2 \cdot \E_{X_s}^2 ( M_{n\delta-s} \cdot s_{h f}(X_{n\delta -s},n\delta))\right) &  \leq \frac{2\|f\|_\infty}{h(x)} \tilde{\E}_x( h(X_s) \tilde{\E}_{X_s}(s_{h f}(X_{n\delta - s},n\delta)) e^{\lambda s}\\
& \leq \frac{2C\|f\|_\infty^2}{h(x)} \tilde{\E}_x\left( h(X_s) \tilde{\E}_{X_s}(1+V(X_{n\delta -s}))\right)e^{-\alpha n\delta}e^{\lambda s}\\
& \leq \frac{2C\|f\|_\infty^2}{h(x)} \tilde{\E}_x\left( h(X_s) (1+K+e^{-\gamma (n\delta - s)}V(X_s))\right)e^{-\alpha n\delta}e^{\lambda s}\\
& \leq \frac{2C(1+K)\|f\|_\infty^2}{h(x)} \tilde{\E}_x\left( h(X_s) (1+V(X_s))\right)e^{-\alpha n\delta}e^{\lambda s}, 
\end{align*} so that
$$
\langle 2 \rangle_{n \delta} \leq 2C(1+K)\|f\|_\infty^2 \overline{\Phi}_x e^{-\alpha n\delta}
$$ and thus $\sum_{n \in \N} \langle 2 \rangle_{n \delta} <+\infty$ since $\overline{\Phi}_x <+\infty$. Hence, we conclude that \eqref{eq:convcsl2} holds and therefore that \eqref{eq:convcs} as well, from which the convergence along lattice times now immediately follows.

To establish the full limit as $t \rightarrow +\infty$, suppose first that $f$ in \eqref{eq:defu} above is the indicator function of some open set $B \in \B_J$. It follows from \eqref{eq:defu} that for any $\delta > 0$ one has the bound
\begin{equation}\label{eq:compu}
U_t^{(x)} \geq \frac{e^{-(r(m_1-1)-\lambda)\delta}}{h(x)(1+\varepsilon)}\sum_{u \in \overline{\xi}^{(x)}_{n\delta}} e^{-(r(m_1-1)-\lambda)n\delta}h(u_{n\delta})\Xi^{(u)}_{B,\delta,\varepsilon}
\end{equation} for $t \in [n\delta,(n+1)\delta)$, where $\Xi^{(u)}_{B,\delta,\varepsilon}$ denotes the indicator function of the event $A^{(u),1}_{B,\delta,\varepsilon} \cap A^{(u),2}_{B,\delta,\varepsilon}$, where
$$
A^{(u),1}_{B,\delta,\varepsilon}:=\left\{ v_t \in B \text{ and }h(v_t) > \frac{1}{1+\varepsilon} h(u_{n\delta}) \text{ for all }v \in \overline{\xi}^{(u)}_t \text{ and }t \in [n\delta,(n+1)\delta)\right\}
$$ and 
$$
A^{(u),2}_{B,\delta,\varepsilon}:=\left\{ \text{The particle }u \in \overline{\xi}^{(x)}_{n\delta} \text{ does not branch in }[n\delta,(n+1)\delta)\right\}.
$$
Now, $\Xi^{(u)}_{B,\delta,\varepsilon}$ is not a function of $u_{n\delta}$ alone as it also depends on the evolution of $\xi^{(u)}$ on the time interval $[n\delta,(n+1)\delta)$, so that our previous analysis for lattice times does not directly apply. However, since $\Xi^{(u)}_{B,\delta,\varepsilon}$ is still bounded, one can adapt the previous analysis to our current setting to show that the expression
$$
\sum_{u \in \overline{\xi}^{(x)}_{n\delta}} e^{-(r(m_1-1)-\lambda)n\delta}h(u_{n\delta})\Xi^{(u)}_{B,\delta,\varepsilon} - \E_x\Bigg( \sum_{u \in \overline{\xi}^{(x)}_{n\delta}} e^{-(r(m_1-1)-\lambda)n\delta}h(u_{n\delta})\Xi^{(u)}_{B,\delta,\varepsilon} \Bigg| \F_{n\delta} \Bigg)
$$ tends to zero almost surely as $n \rightarrow +\infty$. Now, it is easy to see that 
$$
\E_x\Bigg( \sum_{u \in \overline{\xi}^{(x)}_{n\delta}} e^{-(r(m_1-1)-\lambda)n\delta}h(u_{n\delta})\Xi^{(u)}_{B,\delta,\varepsilon} \Bigg| \F_{n\delta} \Bigg) = \sum_{u \in \overline{\xi}^{(x)}_{n\delta}} e^{-(r(m_1-1)-\lambda)n\delta}h(u_{n\delta})\zeta_{B,\delta,\varepsilon}(u_{n\delta})
$$ where for $y \in J$ we define
$$
\zeta_{B,\delta,\varepsilon}(y):= e^{-r(m_1-1)\delta}P_y\left( X_t \in B \text{ and }h(X_t) > \frac{1}{1+\varepsilon}h(y) \text{ for all }\text{ and }t \in [0,\delta)\right).
$$ But by our previous analysis conducted for lattice times we know that 
$$
\frac{1}{h(x)}\sum_{u \in \overline{\xi}^{(x)}_{n\delta}} e^{-(r(m_1-1)-\lambda)n\delta}h(u_{n\delta})\zeta_{B,\delta,\varepsilon}(u_{n\delta}) \overset{as}{\longrightarrow} \nu(h \zeta_{B,\delta,\varepsilon}) \cdot D^{(x)}_\infty
$$ so that by \eqref{eq:compu} we conclude that $P$-almost surely
\begin{equation}\label{eq:compu2}
\liminf_{t \rightarrow +\infty} U^{(x)}_t \geq \frac{e^{-(r(m_1-1)-\lambda)\delta}}{1+\varepsilon} \nu(h\zeta_{B,\delta,\varepsilon}) \cdot D^{(x)}_\infty.
\end{equation} Since $B$ is an open set, $h$ is continuous and $X$ has cadlag trajectories, we have that for any $y \in B$ 
$$
\lim_{\delta \rightarrow 0} P_y\left( X_t \in B \text{ and }h(X_t) > \frac{1}{1+\varepsilon}h(y) \text{ for all }t \in [0,\delta)\right) = 1
$$ so that, by first letting $\delta \rightarrow 0$ and afterwards taking $\varepsilon \rightarrow 0$ in \eqref{eq:compu2}, we have that $P$-almost surely
\begin{equation}\label{eq:compu3}
\liminf_{t \rightarrow +\infty} U^{(x)}_t \geq \nu(h \mathbbm{1}_B) \cdot D^{(x)}_\infty.
\end{equation} 
Now, since $J$ is locally compact and separable, there exists a countable basis $\mathcal{G}:=(G_k)_{k \in \N}$ for its topology which is closed under finite unions. In particular, there exists a full \mbox{$P$-measure set $\tilde{\Omega}^{(x)}$} such that for any $\omega \in \tilde{\Omega}^{(x)}$ the inequality in \eqref{eq:compu3} holds simultaneously for every $G_{k} \in \mathcal{G}$. Hence, since any open set $G \subseteq J$ can be written as the union of an increasing sequence of open sets $(G_{n_k})_{k \in \N}\subseteq \mathcal{G}$, we conclude that for any $\omega \in \tilde{\Omega}^{(x)}$ and $k \in \N$
$$
\liminf_{t \rightarrow +\infty} U^{(x),\mathbbm{1}_G}_t (\omega) \geq \liminf_{t \rightarrow +\infty} U^{(x),\mathbbm{1}_{G_{n_k}}}_t (\omega) \geq \nu(h \mathbbm{1}_{G_{n_k}}) \cdot D^{(x)}_\infty (\omega).
$$ By taking $k \rightarrow +\infty$, we obtain that for any $\omega \in \tilde{\Omega}^{(x)}$
\begin{equation}\label{eq:compu5}
\liminf_{t \rightarrow +\infty} U^{(x),\mathbbm{1}_G}_t (\omega) \geq \nu(h \mathbbm{1}_{G}) \cdot D^{(x)}_\infty (\omega)
\end{equation} holds simultaneously for all open sets $G \subseteq J$. Hence, if for $\omega \in \tilde{\Omega}^{(x)} \cap \{ D^{(x)}_\infty > 0\} \cap \{ D_t^{(x)} \rightarrow D^{(x)}_\infty\}$ and each $t \geq 0$ we define the probability measures $\mu^{(x)}_t(\omega)$ and $\mu$ on $(J,\B_J)$ by the formulas 
$$
\mu_t^{(x)}(\omega)(B):= \frac{U^{(x),\mathbbm{1}_B}_t(\omega)}{D_t^{(x)}(\omega)}\hspace{1cm}\text{ and }\hspace{1cm}\mu(B):=\nu(h \mathbbm{1}_B)
$$ for each $B \in \mathcal{B}_J$ (notice that $\mu_t^{(x)}$ is well-defined because $D_t^{(x)}>0$ on the event $\{D^{(x)}_\infty > 0\}$) then, by Portmanteau's theorem \cite[Theorem 4.25]{kallenberg2002}, \eqref{eq:compu5} tells us that $\mu_t^{(x)}(\omega)$ converges weakly to $\mu$. In particular, since $h$ is strictly positive and continuous on $J$, it follows that if $B \in \mathcal{C}_X$ is such that $\nu(\partial B)=0$ then $\frac{\mathbbm{1}_B}{h}:J \rightarrow \R$ is a bounded and $\B_J$-measurable function whose set of discontinuities is precisely $\partial B$ has $\mu$-null measure, so that by the weak convergence we obtain
$$
\lim_{t \rightarrow +\infty}\mu_t^{(x)}(\omega)\left(\frac{\mathbbm{1}_B}{h}\right) = \mu\left( \frac{\mathbbm{1}_B}{h}\right)=\nu(B).
$$ Since $D^{(x)}_t(\omega) \rightarrow D^{(x)}_\infty(\omega)$ for $\omega \in \tilde{\Omega}^{(x)} \cap \{ D^{(x)}_\infty > 0\} \cap \{ D_t^{(x)} \rightarrow D^{(x)}_\infty\}$, the former implies that
$$
\lim_{t \rightarrow +\infty} U^{(x),\frac{\mathbbm{1}_B}{h}}(\omega) = \nu(B)\cdot D^{(x)}_\infty(\omega).
$$ On the other hand, if $\omega \in \tilde{\Omega}^{(x)} \cap \{ D^{(x)}_\infty = 0\} \cap \{ D_t^{(x)} \rightarrow D^{(x)}_\infty\}$ then clearly as $t \rightarrow +\infty$  we have
$$
\left|U^{(x),\frac{\mathbbm{1}_B}{h}}_t(\omega)\right| \leq \left\|\frac{\mathbbm{1}_B}{h}\right\|_\infty D^{(x)}_t (\omega) \longrightarrow 0 = \nu(B) \cdot D^{(x)}_\infty (\omega).
$$ Thus, since $\Omega^{(x)}:=\tilde{\Omega}^{(x)} \cap \{D^{(x)}_t \rightarrow D^{(x)}_\infty\}$ is a full $P$-measure set by Theorem \ref{theo:main2}, we obtain \eqref{eq:convas}. Finally, \eqref{eq:convas2} follows from \eqref{eq:convas} as in the proof of \eqref{eq:ks2}. This concludes the proof of Theorem \ref{theo:main5}.

\section{Examples and applications : the proofs}\label{sec:examples2}
 
We now discuss how to verify our assumptions for the different examples listed in Section \ref{sec:examples}. Many of the arguments given here can be applied in general also to other systems.

\subsection{Verifying Assumptions \ref{assumpG}}\label{sec:va}

It follows from Proposition \ref{prop:lyapunov1} that perhaps the easiest way to check Assumptions \ref{assumpG} (also (B1)-(B2)) in the $\lambda$-positive case is to find a Lyapunov functional for $X$ which satisfies (V3)-(V4). The next result will give us a simple criterion to check whether a given function $V$ is a Lyapunov functional for $X$.

\begin{proposition}\label{prop:checkv} Let $V:J \rightarrow [0,+\infty)$ be a $\B_J$-measurable function.
		\begin{enumerate}
		\item [$\bullet$] If there exists $t > 0$ such that the transition probabilities of $X_t$ under the measure $\tilde{P}$ have a density $\tilde{f}_t: J \times J \rightarrow \R_{\geq 0}$ with respect to some reference measure $\theta$, i.e.
		$$
		\tilde{\E}_x(g(X_t))=\int_J g(y)\tilde{f}_t(x,y) d\theta(y)
		$$ for all bounded and $\B_J$-measurable $g: J \rightarrow \R$, and this $\tilde{f}_t$ is such that for every $R > 0$ there exists $\beta_R > 0$ and $B_R \in \B_J$ with $\theta(B_R) > 0$ verifying that
		\begin{equation}
		\label{eq:implyv1}
		\inf_{y \in B_R} \tilde{f}_t(x,y) \geq \beta_R
		\end{equation} for all $x \in J$ with $V(x) \leq R$, then assumption (V1) holds.
		
		\item [$\bullet$] If $\tilde{\mathcal{L}}$ denotes the generator of $X$ under the measure $\tilde{P}$, then (V2) is satisfied whenever $V$ verifies the following conditions:
		
		\begin{itemize}
			\item [i.] $V$ is in the domain of $\tilde{\mathcal{L}}$ and for any $x \in J$ and $t > 0$ satisfies Dynkin's formula
			\begin{equation}\label{eq:dfor}
			\tilde{\E}_x(V(X_t)) = V(x) + \int_0^t \tilde{\E}_x(\tilde{\mathcal{L}}(V)(X_s))ds,
			\end{equation} as well as the integrability condition 
			\begin{equation}\label{eq:dfor2}
			\int_0^t \tilde{\E}_x(|\tilde{\mathcal{L}}(V)(X_s)|)ds < +\infty,
			\end{equation}
			\item [ii.] There exist constants $\kappa_1,\kappa_2 > 0$ such that for all $x \in J$
			\begin{equation}\label{eq:v3}
			\tilde{\mathcal{L}}(V)(x) \leq \kappa_1 - \kappa_2 V(x).
			\end{equation}
		\end{itemize} This explains the choice of terminology Lyapunov functional when referring to $V$.
	\end{enumerate}
\end{proposition}

\begin{proof} To obtain the first statement, for any bounded $\B_J$-measurable $f:J \rightarrow \R$ and $x,x' \in J$ such that $\max\{V(x),V(x')\}\leq R$ we may compute
	\begin{align*}
	|\tilde{\E}_x(f(X_t)) -\tilde{\E}_{x'}(f(X_t))| &\leq \|f\|_\infty \int_J |\tilde{f}_t(x,y)-\tilde{f}_t(x',y)|d\theta(y)\\
	& \leq \|f\|_\infty\left( \int_{J-B_R} |\tilde{f}_t(x,y)-\tilde{f}_t(x',y)|d\theta(y) + \int_{B_R}|\tilde{f}_t(x,y)-\tilde{f}_t(x',y)|d\theta(y)\right).
	\end{align*} Now, observe that
	$$
	\int_{J-B_R} |\tilde{f}_t(x,y)-\tilde{f}_t(x',y)|d\theta(y) \leq \int_{J-B_R} \tilde{f}_t(x,y)d\theta(y) + \int_{J-B_R}\tilde{f}_t(x',y)d\theta(y)
	$$ while 
	$$
	\int_{B_R} |\tilde{f}_t(x,y)-\tilde{f}_t(x',y)|d\theta(y) \leq \int_{B_R} (\tilde{f}_t(x,y)-\beta_R)d\theta(y) + \int_{B_R}(\tilde{f}_t(x',y)-\beta_R)d\theta(y),
	$$ since $\inf_{y \in B_R} \tilde{f}_t(z,y) \geq \beta_R$ for $z=x,x'$. From this we conclude that
	$$
	|\tilde{\E}_x(f(X_t))-\tilde{\E}_{x'}(f(X_t))| \leq 2(1-\beta_R\theta(B_R))\|f\|_\infty
	$$ and so (V1) follows. As for the second statement, it follows from \cite[Theorem 4.3, Theorem 6.1]{meyn1993}. We omit the details.
\end{proof}

With this result at hand, we can now proceed on to verify Assumptions \ref{assumpG} for the examples of Section \ref{sec:examples}. In the following, we will often use the fact that for any $f:J \rightarrow \R$ in the domain of $\tilde{\mathcal{L}}$ one has the formula
$$
\tilde{\mathcal{L}}(f)= \frac{\mathcal{L}(h f)}{h} - \lambda f,
$$ which can be deduced at once from the definition of $\tilde{P}$. 

\subsubsection{Subcritical Galton-Watson process} \label{sec:gw2}

By Proposition \ref{prop:lyapunov1}, it will suffice to a find a \mbox{function $V$} satisfying (V1-2-3-4). As it turns out, if $\sigma_\rho^2 <+\infty$ then the eigenfunction $h$  itself will work. Indeed, notice that $V:=h$ trivially satisfies (V3) and (V4), so that it only remains to check (V1)-(V2) with the help of Proposition \ref{prop:checkv}. 

Recall that $h(x)=x$ here and that we may choose the normalization of $\nu$ so that $\nu(h)=1$. Now, to check (V1) we notice that the transition probabilities of $X_t$ under $\tilde{P}$ have a density with respect to the counting measure which is given by 
$$
\tilde{f}_t(x,x'):=\tilde{P}_{x}(X_t = x') = \frac{h(x')}{h(x)}e^{\lambda t}P_x(X_t=x') = \frac{x'}{x}e^{\lambda t}P_x(X_t = x')
$$ for any $x,x' \in J$. Now, observe that for any $x \in J$ one has the lower bound
$$
P_x(X_1 = 1 ) \geq \frac{e^{-1}}{(x-1)!}[\rho(-1)]^{x-1}
$$ which corresponds to the probability that there were exactly $x-1$ deaths and no births during the time interval $[0,1]$. In particular, this yields that
$$
\tilde{f}_1(x,1) \geq \frac{1}{R!}e^{\lambda-1} [\rho(-1)]^{R-1} > 0
$$ for all $x \in \{1,\dots,R\}$, which implies \eqref{eq:implyv1} and therefore (V1) as well.

On the other hand, it is easy to see that if $\sigma^2_\rho < +\infty$ then $h$ is in the domain of $\tilde{\mathcal{L}}$ and satisfies
$$
\tilde{\mathcal{L}}(h)(x)= \frac{\L(h^2)(x)}{h(x)} + \lambda h(x)= \sum_{y} \rho(y)((x+y)^2-x^2) - x \sum_y y\rho(y)= \sigma_\rho - \lambda h(x)
$$ for all $x \in \N$, from where (V2) immediately follows by Proposition \ref{prop:checkv}. 

This shows that Assumptions \ref{assumpG} hold whenever $\sigma^2_\rho <+\infty$. Moreover, since $\inf_{x \in J}h(x) > 0$, one has that in fact $\mathcal{C}_X=\B_J$. On the other hand, it is not difficult to check that $\E_x(M_t^2)=+\infty$ for every $t > 0$ if $\sigma_\rho = +\infty$. Thus, it follows that in this case the sequences $D^{(x)}$ and $W^{(x)}(B,B')$ are not bounded in $L^2$ for any $x \in \N$. 

\subsubsection{Subcritical contact process on $\Z^d$} As for the previous model, it will suffice to show that $h$ is a Lyapunov functional for $X$. In this case $h$ is not explicitly known, although it is shown in \cite[Lemma 2.9]{sturm2014} that $h(\zeta)$ is equivalent to the number of infected sites in $\zeta$, i.e. there exist two constants $c_1, c_2 >0$ such that
\begin{equation}\label{ineq:CP}
c_1 |\zeta| \leq h(\zeta) \leq c_2 |\zeta|,
\end{equation} for any $\zeta \in J$, where $|\zeta|$ stands for the cardinality of any $\sigma \in \mathcal{P}_f(\Z^d)$ such that $\langle \sigma \rangle = \zeta$. 

The proof that (V1) holds here is analogous to that of Section \ref{sec:gw2}, so we omit the details and move on to show (V2). Note that if $\zeta=\langle \sigma \rangle$ for some nonempty $\sigma \in \mathcal{P}_f(\Z^d)$ then
$$
\mathcal{L}(h^2)(\zeta) = \sum_{x \in \sigma} h^2(\langle \sigma - \{x\}\rangle) - h^2(\zeta) 
+ \sum_{y \sim \sigma} \gamma |\{x \in \sigma : |x-y|_1 = 1\}|(h^2(\langle \sigma \cup \{y\}\rangle)-h^2(\zeta))
$$ 
where $y \sim \sigma$ if $|y-x|_1=1$ for some $x \in \sigma$. Using that $a^2-b^2=(a-b)(a+b)$ for any $a,b \in \R$ and recalling that $h$ is the eigenvector of $\mathcal{L}$ associated to $-\lambda$, a straightforward computation yields 
\begin{align*}
\tilde{\mathcal{L}}(h)(\zeta)&= \sum_{x \in \sigma} \left(h(\langle \sigma - \{x\}\rangle) - h(\zeta)\right)\left(\frac{h(\langle \sigma-\{x\}\rangle)}{h(\zeta)}\right)\\
& + \sum_{y \sim \sigma} \gamma |\{x \in \sigma : |x-y|_1 = 1\}|\left(h(\langle\sigma \cup \{y\}\rangle)-h(\zeta)\right) \left( \frac{h(\langle \sigma \cup \{y\}\rangle)}{h(\zeta)}\right).
\end{align*} Now, observe that for any $\zeta=\langle \sigma \rangle \in J$ and $x \in \Z^d$ we have that
\begin{equation}\label{eq:concomp}
h(\zeta) \leq h(\langle \sigma \cup \{x\} \rangle) \leq h( \zeta) + h( \langle \{0\} \rangle).
\end{equation} Indeed, we can couple $Y^{(\sigma\cup\{x\})}$ with the pair $(Y^{(\sigma)},Y^{(\{x\})})$ so that $
Y^{(\sigma\cup\{x\})}_t = Y^{(\sigma)}_t \cup Y^{(x)}_t
$ holds for all $t \geq 0$, which gives the inequalities 
$$
P_\zeta( X_t \neq \emptyset ) \leq P_{\langle \sigma \cup \{x\}\rangle}( X_t \neq \emptyset) \leq P_\zeta( X_t \neq \emptyset) + P_{\langle \{0\} \rangle}( X_t \neq \emptyset)
$$ from where \eqref{eq:concomp} immediately follows since 
$$
h(\zeta)=\lim_{t \rightarrow +\infty} e^{\lambda t} P_\zeta(X_t \neq \emptyset )
$$ for any $\zeta \in J$ by the $\lambda$-positivity of $X$, provided that we normalize $\nu$ to be a probability measure and choose the multiple of $h$ so that $\nu(h)=1$. Hence, we conclude that
$$
\tilde{\mathcal{L}}(h)(\zeta) \leq -\lambda h(\zeta)+\frac{h(\langle \{0\}\rangle)}{h(\zeta)} (\mathcal{L}^-(h)(\zeta)+\mathcal{L}^+(h)(\zeta)),
$$ where
$$
\mathcal{L}^-(h)(\zeta):= \sum_{x \in \tilde{\sigma}} \left(h(\zeta) - h(\langle \tilde{\sigma} - \{x\}\rangle)\right) \leq h(\langle \{0\}\rangle)|\zeta|
$$ and
$$
\mathcal{L}^+(h)(\zeta):=\sum_{y \sim \tilde{\sigma}} \gamma |\{x \in \tilde{\sigma} : |x-y|_1 = 1\}|\left(h(\langle\tilde{\sigma} + \{y\}\rangle)-h(\zeta)\right) \leq \gamma(2d)^2 h(\langle \{0\}\rangle) |\zeta|.
$$ Therefore, by combining all the previous estimates with \eqref{ineq:CP}, we conclude that 
$$
\tilde{\mathcal{L}}(h)( \zeta ) \leq -\lambda h(\zeta) +\frac{c_2^2}{c_1}(1+\gamma(2d)^2)
$$ from which (V2) now follows. In particular, Assumptions \ref{assumpG} are satisfied with $\mathcal{C}_X=\B_J$ as before. 

\subsubsection{Recurrent Ornstein-Uhlenbeck process}\label{sec:rou2} In this case,showing that $h$ is a Lyapunov functional is more involved. Indeed, a direct computation shows that 
$$
\tilde{\mathcal{L}}(h)(x)= \sqrt{\frac{4\lambda}{\pi}} \cdot \frac{1}{x}-\lambda h(x)
$$ so that, in particular, \eqref{eq:v3} does not hold and therefore we cannot use Proposition \ref{prop:checkv} to show that (V2) holds for $V:=h$. This is true, however, but it requires a more refined argument. 
Indeed, first we observe that by \cite[Remark 2.3]{alili2005representations} we have  
$$
P_x( X_t > 0) = \int_{\tau(t)}^{+\infty} \frac{x}{\sqrt{2\pi s^3}} e^{-\frac{x^2}{2s}}ds
$$
for all $x \in J$ and $t \geq 0$, where 
$$
\tau(t):= \frac{e^{2\lambda t}-1}{2\lambda}.
$$ Using the bound $e^{-\frac{x^2}{2\tau(t)}} \leq e^{-\frac{x^2}{2s}} \leq  1$ for all $s \in [\tau(t),+\infty)$, we obtain
$$
\frac{x}{\sqrt{2\pi}}\cdot e^{-\frac{x^2}{2\tau(t)}} \cdot \frac{2}{\sqrt{\tau(t)}} \leq P_x(X_t >0) \leq \frac{x}{\sqrt{2\pi}}\cdot \frac{2}{\sqrt{\tau(t)}},
$$ from which a straightforward calculation yields that
\begin{equation}\label{eq:roup}
P_x(X_t >0) = h(x)e^{-\lambda t}(1+s_J(x,t))
\end{equation} where $s_J$ satisfies the bounds
\begin{equation}\label{eq:roup2}
\frac{e^{-\frac{x^2}{2\tau(t)}}}{\sqrt{1-e^{-2\lambda t}}}-1 \leq s_J (x,t) \leq \frac{1}{\sqrt{1-e^{-2\lambda t}}}-1.
\end{equation} Now, by the product rule (which can be easily justified here), from \eqref{eq:roup} and \eqref{eq:roup2} we obtain 
$$
\frac{d \tilde{\E}_x(e^{\lambda t}h(X_t))}{dt} = e^{\lambda t}(\tilde{\E}_x(\tilde{\mathcal{L}}(h)(X_t)) + \lambda \tilde{\E}_x(h(X_t))) = \frac{4\lambda e^{\lambda t}}{\pi} \cdot \tilde{\E}_x\left( \frac{\mathbbm{1}_{\{X_t > 0\}}}{h(X_t)}\right) \leq \frac{4\lambda}{\pi} \frac{e^{\lambda t}}{\sqrt{1-e^{-2\lambda t}}},
$$ so that 
$$
e^{\lambda t}\tilde{\E}_x(h(X_t)) = h(x) + \int_0^t \frac{d \tilde{\E}_x(e^{\lambda t}h(X_t))}{dt}\bigg|_{t=s}ds \leq h(x) + \frac{4\lambda}{\pi} \int_0^t \frac{e^{\lambda s}}{\sqrt{1-e^{-2\lambda s}}}ds.
$$ A straightforward calculation then shows that there exists a constant $c_\lambda > 0$ such that for \mbox{any $t \geq 0$}
$$
\int_0^t \frac{e^{\lambda s}}{\sqrt{1-e^{-2\lambda s}}}ds \leq c_\lambda e^{\lambda t},
$$ so that
$$
\tilde{\E}_x(h(X_t)) \leq e^{-\lambda t}h(x) + \frac{4\lambda}{\pi}c_\lambda
$$ and hence (V2) is satisfied. On the other hand, it follows from \cite{lladser2000} that on the event $\{X_t^{(x)} > 0\}$ the random variable $X_t^{(x)}$ has the following density with respect to Lebesgue measure:
\begin{equation}\label{eq:densityrou}
f_t(x,y):=P_x(X_t \in dy , T > t) = \sqrt{\frac{2}{\pi e^{-2\lambda t}\tau(t)}}\exp\left(- \frac{x^2}{2\tau(t)}- \frac{y^2}{2e^{-2\lambda t}\tau(t)}\right)\sinh \left(\frac{xy}{e^{-\lambda t}\tau(t)}\right).
\end{equation} In particular, the distribution of $X_t$ under the change of measure $\tilde{P}$ has density $\tilde{f}_t$ given by 
\begin{equation}\label{eq:eeff}
\tilde{f}_t(x,y):= \tilde{P}_x(X_t \in dy)= \frac{h(y)}{h(x)}e^{\lambda t}f_t(x,y).
\end{equation} From here it is easy to check that for any $R > 0$ there exist  $\beta_R > 0$ and some small set $B_R \in \B_J$ such that
$$
\inf_{y \in B_R} \tilde{f}_t(x,y) \geq \beta_R
$$ for all $x \in J$ with $V(x) \leq R$, from which (V1) follows and thus $h$ is a Lyapunov functional for $X$. 
By Proposition \ref{prop:lyapunov1} we obtain that Assumptions \ref{assumpG} hold for $\mathcal{C}_X$ given by 
$$
\mathcal{C}_X:=\left\{ B \in \B_J : \left\|\frac{\mathbbm{1}_B}{h}\right\|_\infty < +\infty\right\} = \left\{ B \in \B_J : \inf B > 0\right\}.
$$ However, it is possible to show that (A2) holds in fact for \textit{every} $B \in \B_J$. Indeed, \eqref{eq:roup}-\eqref{eq:roup2} together imply that (A2) holds for $B=J$. To obtain the same for any $B \in \B_J$, by the $\lambda$-positivity we will only need to check that (A2-iii) and (A2-iv) hold for any such $B$. But notice that for every $n \in \N$ one has the bound
$$
|s_B(x,t)| \leq |s_J(x,t)|+|s_{J_n}(x,t)|+|s_{B\cap J_n}(x,t)| + 2|\nu(J) - \nu(J_n)|.
$$ from which (A2-iii-iv) immediately follow, since all sets on the right-hand side satisfy (A2). 

\subsubsection{Transient Ornstein-Uhlenbeck process}\label{sec:tou2} Notice that here $h$ cannot be a Lyapunov functional due to the fact that it is bounded, so that (V1) fails to hold. Instead, we shall show that $V(x):=x^2$ is indeed a Lyapunov functional for $X$. Since $V$ also satisfies (V3)-(V4), from Proposition \ref{prop:lyapunov1} we will get that Assumptions \ref{assumpG} hold.  

Now, it is straightforward to check that the generator $\tilde{\mathcal{L}}$ of $X$ under the measure $\tilde{P}$ is given by 
$$
\tilde{\mathcal{L}}(f)(x):= \frac{1}{2}\sigma^2f''(x) - \lambda x f'(x)
$$ for any $f \in C^2(\R)$ and $x \in \R$. In particular, under the measure $\tilde{P}$ the random variable $X^{(x)}_t$ has a density $\tilde{f}_t$ with respect to Lebesgue measure given by 
\begin{equation}\label{eq:dentou}
\tilde{f}_t(x,y):= \frac{1}{\sqrt{2\pi\beta^{-1}(1-e^{-2\beta t})}}\exp\left\{- \frac{(y-xe^{-\beta t})^2}{\beta^{-1}(1-e^{-2\beta t})}\right\},
\end{equation} with $\beta:=\frac{\lambda}{\sigma^2}$, from where \eqref{eq:implyv1} follows and, hence, also (V1). On the other hand, it is easy to see that 
$$
\tilde{\mathcal{L}}(V)=\sigma^2 -2\lambda V
$$ so that (V2) is also satisfied, and thus $V$ is indeed a Lyapunov functional for $X$. 

\subsubsection{Brownian motion with drift killed at the origin}\label{sec:bbm2} In \cite{polak2012} it is shown that, if we take $\lambda:=\frac{c^2}{2}$ and choose the normalizations 
$$
h(x):= \frac{1}{\sqrt{2\pi \lambda^2}}x e^{cx} \hspace{1cm}\text{ and }\hspace{1cm}\frac{d\nu}{dx}(x):=2\lambda x e^{-cx}\mathbbm{1}_{(0,+\infty)}(x)
$$ (notice that, in particular, $\nu$ is a probability measure with this normalization), then $X$ satisfies (A2) for all $B \in \B_{(0,+\infty)}$, with $p$ given by $p(t):=t^{-\frac{3}{2}}$. Moreover, it is shown that for any $x,t > 0$ we have 
\begin{equation}\label{eq:densitybm}
P_x ( X_t \in dy , X_t > 0) = \frac{e^{cx}}{\sqrt{2\pi t}} e^{-\lambda t} e^{-cy} \left[ \exp{\left( - \frac{(x-y)^2}{2t }\right)}-\exp{\left( - \frac{(x+y)^2}{2t}\right)}\right]\mathbbm{1}_{(0,+\infty)}(y),
\end{equation} so that we may compute
$$
\E_x(h^2(X_t))= \frac{e^{cx}}{\sqrt{2\pi t}} e^{-\lambda t} \int_0^y h^2(y)e^{-cy} \left[ \exp{\left( - \frac{(x-y)^2}{2t }\right)}-\exp{\left( - \frac{(x+y)^2}{2t}\right)}\right]dy.
$$
A straightforward (but tedious) calculation shows that 
\begin{equation} \label{eq:bbmphi}
\E_x(h^2(X_t))=\frac{1}{2\pi\lambda}e^{cx}\left( F(x)-F(-x)\right),
\end{equation}
where 
$$
F(x):=\sqrt{2\pi}e^{cx}[(ct+x)^2+t] - \frac{e^{-\frac{x^2}{2t}-\lambda t}}{c\sqrt{t}+\frac{x}{\sqrt{t}}}\left([(ct+x)^2+t](1+o_x(1))+(ct+x)^2\right)
$$ and $o_x(1) \rightarrow 0$ as $t \rightarrow +\infty$ uniformly over compact sets of $x \in \R_+$. In particular, we get (A1), since it is straightforward to check that $-\lambda$ is an eigenvalue of $\mathcal{L}$ with associated eigenfunction $h$. On the other hand, it follows from \eqref{eq:bbmphi} that for $x > 0$ we have
\begin{equation} \label{eq:bbmphi2}
\lim_{t \rightarrow +\infty} \frac{\E_x(h^2(X_t))}{t^2}=\frac{2}{\sqrt{2\pi}}(e^{2cx}-1).
\end{equation} In particular, we conclude that 
$$
\Phi_x < +\infty \Longleftrightarrow r(m_1-1)> 2\lambda,
$$ so that $L^2$-convergence does not hold for this model in the entire supercritical region $r(m_1-1)> \lambda$.

\subsection{Verifying Assumptions \ref{assumpG2}}

By Proposition \ref{prop:lyapunov1} and the analysis of Section \ref{sec:va}, it follows that all $\lambda$-positive processes of Section \ref{sec:examples} satisfy (B1) and (B2), so that it only remains to check these assumptions for the Brownian motion with drift killed at the origin. However, it follows from Section \ref{sec:bbm2} that (B1) holds if $r(m_1-1) >2\lambda$ and (B2) always holds since $\mathcal{C}_X=\B_J$. 

 As for assumption (B3), we noted earlier that, in the case of a countable state space $J$, (B3) is weaker than standard irreducibility. Indeed, given $x,x' \in J$ and $B \in \B_J$, by the Markov property one has that
$$
P_x(X_{n+1} \in B) \geq P_x( X_n = x')P_{x'}(X_1 \in B)
$$ and thus, since $P_x(X_n = x') > 0$ by irreducibility, (B3) immediately follows. In particular, (B3) is satisfied for the models of Sections \ref{sec:gw} and \ref{sec:cp} since they are both (standard) irreducible.  

For the remaining examples having an uncountable state space, we can use the following result. 

\begin{proposition}\label{prop:b2} Assumption (B3) holds if $X$ satisfies the following condition: 
	\begin{enumerate}
		\item [(D).] There exists some reference measure $\theta$ on $(J,\B_J)$ such that for any $x \in J$ and every $t > 0$ the conditional distribution $P_{x|t}$ of $X^{(x)}_t$  on the event $\{ X_t \in J \}$ is absolutely continuous with respect to $\theta$ and has a density $f_{x|t}$ which is strictly positive on $J$.
	\end{enumerate}
\end{proposition}

\begin{proof} If (D) holds then for any pair $x,x' \in J$ and $t > 0$ we have that $P_{x'|1} \ll P_{x,t}$ with density given by 
	$$
	f(y):=\frac{f_{x'|1}(y)}{f_{x|t}(y)}.
	$$ In particular, (B3) follows at once.
\end{proof}

We see from Section \ref{sec:va} that the models in Sections \ref{sec:tou}, \ref{sec:rou} and \ref{sec:bbm} all satisfy (D) for $\theta$ given by the Lebesgue measure on $J$. In particular, we get that (B3) holds in all these examples. Finally, we note that Proposition \ref{prop:b2} can also be used to establish (B3) for the models of Sections \ref{sec:gw}-\ref{sec:cp}, since they both satisfy (D) with $\theta$ given by the counting measure.  

\subsection{Verifying strong supercriticality} \label{sec:supercritical}

As a first criterion for proving strong supercriticality, we observe that if Assumptions \ref{assumpG2} are satisfied and $\sup_{x \in J} \Phi_x < +\infty$ then $\xi$ is automatically strongly supercritical. Indeed, if $\sup_{x \in J} \Phi_x < +\infty$ then $\tilde{J}_n \equiv J$ for all $n$ sufficiently large, so that 
$$
\Theta=\left\{ \limsup_{t \rightarrow +\infty} \xi_t(\tilde{J}_n) > 0\right\}
$$ for any such $n \in \N$ and therefore, by Proposition \ref{prop:ss1}, we conclude that $\xi$ is strongly supercritical (recall that $\Phi_x < +\infty$ implies that $\sigma(x)<1$, and thus that $\xi^{(x)}$ is supercritical in the first place).

This simple criterion is already sufficient to establish strong supercriticality in the examples of Sections \ref{sec:gw} and \ref{sec:cp}. Indeed, since for both systems we have that $h$ is a Lyapunov functional for $X$ and $\inf_{x \in J} h(x) > 0$, it follows from \eqref{eq:compvc} that $\sup_{x \in J} \Phi_x < +\infty$ and so both dynamics are strongly supercritical. 

In the remaining systems of Section \ref{sec:examples}, however, we always have that $\sup_{x \in J} \Phi_x = +\infty$, so that establishing strong supercriticality (or its lack of) will require a more refined argument.

\subsubsection{Studying the rightmost particle} For the systems of Sections \ref{sec:rou}- \ref{sec:bbm}, this argument consists of studying the behavior of the rightmost particle. Indeed, assume that (B1) holds and notice then that by \eqref{eq:compvc}-\eqref{eq:bbmphi2} it follows that, for both systems, we have $\sup_{x \in B} \Phi_x < +\infty$ for all $B \in \B_{(0,+\infty)}$ which are bounded away from $0$. Therefore, the only way in which the event $\Gamma^{(x)}$ can occur is if all particles in $\xi$ tend \textit{uniformly} to zero as $t \rightarrow +\infty$. In particular, if we show that 
\begin{equation}\label{eq:maximo}
\limsup_{t \rightarrow +\infty} \left[\max_{u \in \overline{\xi}^{(x)}_t} u_t\right]> 0
\end{equation} almost surely on the event of non-extinction, then the strong supercriticality immediately follows. For the system of Section \ref{sec:bbm}, \eqref{eq:maximo} becomes a direct consequence of \cite[Theorem 9]{harris2006v1} and, moreover, the argument there can be adapted to hold also for the system of Section \ref{sec:rou}, so that both models turn out to be strongly supercritical. For completeness, we sketch the proof here for the latter below and refer to \cite{harris2006v1} for further details.

For each $x,t > 0$ let $\F_t^{(x)}:=\sigma(\xi^{(x)}_s : 0 \leq s \leq t)$ denote the $\sigma$-algebra generated by the dynamics until time $t$ and define $T^{(x)}_0 := \inf\{ t \geq 0 : X^{(x)}_t = 0\}$. Note that one has the inequality
$$
P_x( \Theta^c | \F_t) \geq \prod_{u \in \overline{\xi}_t^{(x)}} P_{u_t}( T_0 < E_r ) = \prod_{u \in \overline{\xi}_t^{(x)}} \E_{u_t}\left( e^{-rT_0} \right)
$$ where $E_r$ is an independent exponential random variable with parameter $r$, since extinction would follow if all particles in $\overline{\xi}^{(x)}_t$ hit the origin before branching. By dominating each $X^{(u_t)}$ from above by a Brownian motion started at $u_t$, standard expressions for the one-side exit problem of Brownian motion yield that
$$
P_x( \Theta^c | \F_t) \geq \prod_{u \in \overline{\xi}_t^{(x)}} e^{-\sqrt{2r}u_t} = \exp\left( \sqrt{2r}\sum_{u \in \overline{\xi}^{(x)}_t} u_t\right).
$$ Since clearly $\lim_{t \rightarrow +\infty} P_x(\Theta^c|\F_t)=0$ almost surely on the event of non-extinction, it follows that on this event one has
\begin{equation} \label{eq:inf}
\lim_{t \rightarrow +\infty} \sum_{u \in \overline{\xi}^{(x)}_t} u_t = +\infty.
\end{equation} Now, if for each $M > 0$ we denote by $\Upsilon^{(x)}_M$ the event that the branching dynamics $\xi^{(x)}$ is contained for all times on the strip $(0,M)$, then it is not hard to show that on the event $\{ \max_{u \in \overline{\xi}^{(x)}_t} u_t < M\}$ one has
$$
P_x(\Upsilon_M|\F_t) \leq \prod_{u \in \overline{\xi}^{(x)}_t} P_{u_t}( T_0 < T_M)
$$ where $T^{(x)}_M:=\inf\{ t \geq 0 : X^{(x)}_t = M\}$. By dominating each $X^{(u_t)}$ on the event $\{T_0^{(u_t)} < T^{(u_t)}_M\}$ from below by a Brownian motion with drift $-M$ started at $u_t$, again classical results but now for the two-side exit problem for Brownian motion yield that on $\{ \max_{u \in \overline{\xi}^{(x)}_t} u_t < M\}$ we have
$$
P_x(\Upsilon_M|\F_t) \leq \exp\left( - c_{\lambda,M} \sum_{u \in \overline{\xi}^{(x)}_t} u_t\right)
$$ for some constant $c_{\lambda,M} > 0$, from where using \eqref{eq:inf} one can deduce that $P_x(\Upsilon_M)=0$ for all $M$. Thus, from here it is easy to show that 
$$
P_x\left( \limsup_{t \rightarrow +\infty} \left[ \max_{u \in \overline{\xi}_t} u_t\right]\geq M \right) = P_x(\Theta)
$$ for all $M > 0$, from where \eqref{eq:maximo} follows.

\subsubsection{The dynamics $\xi$ associated to the transient O-U of Section \ref{sec:tou} is not strongly supercritical} To conclude our analysis of strong supercriticality for the systems of Section \ref{sec:examples}, we now prove this. Notice that, since Assumptions \ref{assumpG2} are all satisfied for this model, by Theorem \ref{theo:main3} it will suffice to exhibit a fixed point $g$ of the moment-generating operator $G$ which is different from $\eta$ and $\mathbf{1}$. Furthermore, since there is no absorption here, we have $\eta \equiv \mathbf{0}$ and therefore it will suffice to find a fixed point $g$ and $x \in \R$ such that $0 < g(x) < 1$. We will choose $g$ as 
$$
g(x):= P_x \left( \limsup_{t \rightarrow +\infty} \xi_t((-\infty,0]) = 0 \right),
$$ which can be checked to be a fixed point as in Proposition \ref{prop:fp}, and then show that $0< g(x)  <1$ for all $x > 0$ sufficiently large. 

Now, on the one hand, observe that by Theorem \ref{theo:main} we have that $\limsup_{t \rightarrow +\infty} \xi^{(x)}_t((-\infty,0]) > 0$ on the event $\{D^{(x)}_\infty > 0\}$ which has positive probability, so that necessarily $g(x) < 1$ for all $x \in \R$. To see that $g(x) > 0$ holds for every $x > 0$ sufficiently large, we will use a coupling argument. Indeed, recall that for each $t \geq 0$ one has the representation formula 
$$
X^{(x)}_t = x + \lambda \int_0^t X_s dt + \sigma B_t 
$$ so that if for $\gamma > 0$ we set 
$$
Y^{(x),\gamma}_t = x + \gamma t + \sigma B_t
$$ then $X^{(x)}_t \geq Y^{(x),\gamma}_t$ for all $t < T^{(x)}_{\frac{\gamma}{\lambda}}:=\inf\{ s \geq 0 : X^{(x)}_s \leq \frac{\gamma}{\lambda}\}$, with the convention that $\inf \emptyset = +\infty$. Now, let $\xi^{(x)}$ and $\xi^{(x),\gamma}$ denote the branching dynamics starting at $x$ associated to $X$ and $Y$, respectively. It follows from \cite[Theorem 4.17]{bocharov2012branching} that there exist $\gamma^*,x^* > 0$ such that for all $x \geq x^*$ one has
\begin{equation}
\label{eq:event1}
P_x\left( \xi^{\gamma^*}_t\left(\left(-\infty,\frac{x}{2}\right]\right) = 0 \text{ for all }t \right) = P_x\left( \min_{u \in \overline{\xi}^{\gamma^*}_t} u_t > \frac{x}{2} \text{ for all }t \right) > \frac{1}{2}.
\end{equation} Hence, if $x \geq \max\{x^*,2\frac{\gamma^*}{\lambda}\}$ then one can couple the dynamics $\xi^{(x)}$ and $\xi^{(x),\gamma^*}$ in such a way that on the event from \eqref{eq:event1} one has 
$$
\min_{u \in \overline{\xi}_t^{(x)}} u_t \geq \min_{u \in \overline{\xi}^{(x),\gamma^*}_t} u_t > \frac{\gamma^*}{\lambda} > 0
$$ for all $t \geq 0$. In particular, we obtain that $g(x) > \frac{1}{2}$ for any such $x$ and hence our claim is proved. To conclude, we observe that Theorem \ref{theo:main3} then implies that $\sigma(x) \neq \eta(x) = 0$ for all $x \in \R$ and, since also $\sigma(x) < 1$ for all $x \in \R$ by Theorem \ref{theo:main}, we obtain that $0 < \sigma(x) < 1$ for all $x \in \R$. 

\subsection{Verifying the $\overline{\Phi}_x <+\infty$ condition} 

Finally, we mention that the $\overline{\Phi}_x < +\infty$ condition will, in general, not follow from our previous assumptions. Still, if $X$ admits a Lyapunov functional $V$ satisfying (V4) and such that there exists a constant $C > 0$ for which
$$
\tilde{\E}_x(V^2(X_t)) \leq C(1+t+V^2(x))
$$ holds for all $t \geq 0$, then the reader may check by a direct computation that $\overline{\Phi}_x < +\infty$. Fortunately, this is the case in all our examples since for each of them we have, with $V$ chosen as in Section \ref{sec:examples2}, that $V^2$ is in the domain of $\tilde{\mathcal{L}}$, satisfies Dynkin's formula
		$$
		\tilde{\E}_x(V^2(X_t)) = V^2(x) + \int_0^t \tilde{\E}_x(\tilde{\mathcal{L}}(V^2)(X_s))ds
		$$ for all $t \geq 0$ and $x \in J$,
		and verifies $\tilde{\mathcal{L}}(V^2)(x) \leq C'(1+V(x))$ for some constant $C' > 0$ (although for the subcritical Galton-Watson process of Section \ref{sec:gw} we need to ask for the offspring distribution $\rho$ to have a finite third moment in order to have $V^2$ in the domain of $\tilde{\mathcal{L}}$). We omit the details, which are straightforward.

\subsection{Relaxing Assumptions \ref{assumpG} and the proof of Theorem \ref{theo:ergod}}\label{sec:ergod2}

It follows from the proof given in Section \ref{sec:proof} that Theorem \ref{theo:main} still holds if one exchanges the sequence $(J_n)_{n \in \N}$ defined in Assumptions \ref{assumpG} for any other increasing sequence of subsets of $J$ satisfying:
\begin{enumerate}
	\item [J1.] The union of all $J_n$ coincides with $J$. 
	\item [J2.] For any $B \in \mathcal{C}_X$ the asymptotic formula \eqref{A2} holds and the error term $s_B(\cdot,t)$ converges to zero as $t$ tends to infinity uniformly over each $J_n$.
	\item [J3.] The contents of Lemma \ref{lema:A3} remain true, i.e. for any $T > 0$ one has
	$$
	\lim_{n \rightarrow +\infty} \left[\sup_{t \in [0,T]} \E_x \left( M_t^2 \mathbbm{1}_{\{X_t \notin J_n\}}\right)\right] = 0.
	$$
\end{enumerate}
Furthermore, Theorem \ref{theo:main3} still holds if, in addition, we ask that $\inf_{x \in J_n} h(x) > 0$ for each $n$ and Theorem \ref{theo:main5} also remains true if we require that $\sup_{x \in J_n} V(x) <+\infty$ for each $n$ instead of (V3). 

For the particular case of ergodic motions considered in Section \ref{sec:cwa}, it is not hard to show that any increasing sequence $(J_n)_{n \in \N} \subseteq \B_J$ of open sets with compact closure such that their union gives all of $J$ (such a sequence exists because $J$ is locally compact and separable) satisfies (J1-2-3) for the class $\mathcal{C}_X$ given by \eqref{eq:defctilde}. Hence, we obtain that Theorem \ref{theo:main} holds in this case. Furthermore, that $\inf_{x \in J_n} h(x) > 0$ for each $n$ is immediate in this context and also that $\sup_{x \in J_n} V(x)<+\infty$ whenever $V$ is bounded over compact subsets of $J$. Thus, we also obtain Theorems \ref{theo:main3} and \ref{theo:main5}, so that Theorem \ref{theo:ergod} now follows.    

\section{Conclusions and future work}
Apart from a theoretical interest, our results have of course strong implications for the possibility of simulating quasi-stationary distributions whose closed form might be unknown. In that respect, $L^2$-convergence becomes a crucial feature to look at since it controls the variance of the outcome of the simulation.
Our results also open up the possibility of investigating more refined branching mechanisms in order to speed up simulations, but this is left for future research. Also, we plan to characterize the conditions for $L^2$-convergence in terms of $r,m$ and $\lambda$ as in Theorem \ref{theo:bbm} for more general L\'evy processes in a future work. Finally, the characterization of almost sure convergence beyond the $\lambda$-positive setting is also a thrilling open question.

\section*{Acknowledgements:}
Matthieu Jonckheere would like to warmly thank Elie A\"idekon, Julien Beresticky, Simon Harris, Pablo Groisman and Pascal Maillard for many fruitful discussions on branching dynamics. 

Santiago Saglietti would like to also thank Maria Eul\'alia Vares for useful suggestions on how to improve this manuscript and Bastien Mallein for very helpful discussions on these topics.

\bibliographystyle{plain}
\bibliography{biblio4}
 
\end{document}